\documentclass[11pt, journal, onecolumn]{IEEEtran}

\usepackage{amsmath,amssymb,graphicx,paralist,subfigure,pdfpages,cite,algorithm,algpseudocode,paralist,url}
\usepackage{amsthm}
\usepackage{verbatim}
\usepackage{algorithm}
\usepackage{algpseudocode}
\newtheorem{lem}{Lemma} 
\newtheorem{theorem}{Theorem}

\newtheorem{assump}{Assumption}

\def\mc{\mathcal}
\def\mb{\mathbf}
\def\mbb{\mathbb}

\IEEEoverridecommandlockouts

\renewcommand{\baselinestretch}{1.5}

\def\mbb{\mathbb}
\def\mb{\mathbf}
\def\mc{\mathcal}

\def\ol{\overline}

\def\bds{\boldsymbol}
\newcommand{\mn}[1]{{\left\vert\kern-0.25ex\left\vert\kern-0.25ex\left\vert\kern0.3ex #1 
		\kern0.3ex\right\vert\kern-0.25ex\right\vert\kern-0.25ex\right\vert}}

\begin{document}
\title{Variance-Reduced Decentralized Stochastic Optimization with Gradient Tracking -- \\Part I:~\textbf{\texttt{GT-SAGA}}}
	\author{
		Ran Xin, Usman A. Khan, and Soummya Kar
		\thanks{
			RX and SK are with the Department of Electrical and Computer Engineering, Carnegie Mellon University, Pittsburgh, PA; {\texttt{ranx@andrew.cmu.edu, soummyak@andrew.cmu.edu}}. UAK is with the Department of Electrical and Computer Engineering, Tufts University, Medford, MA; {\texttt{khan@ece.tufts.edu}}.
			The work of RX and SK has been partially supported by NSF under grant CCF-1513936.
			The work of UAK has been partially supported by NSF under grants CCF-1350264 and CMMI-1903972. 
		}
	}
	\maketitle
	
	\begin{abstract}
		In this paper, we study decentralized empirical risk minimization problems, where the goal to minimize a finite-sum of smooth and strongly-convex functions available over a network of nodes. We propose~\textbf{\texttt{GT-SAGA}}, a stochastic first-order algorithm based on decentralized stochastic gradient tracking methods (GT)~\cite{DSGT_Pu,DSGT_Xin} and a variance-reduction technique called SAGA~\cite{SAGA}. We demonstrate various trade-offs and discuss scenarios in which~\textbf{\texttt{GT-SAGA}} achieves superior performance (in terms of the number of local gradient computations required) with respect to existing decentralized schemes. 
	\end{abstract}

	\noindent\fbox{%
    \parbox{\textwidth}{%
  	\begin{center}
	    {\color{red}\textbf{This is a preliminary version of the paper~\url{https://arxiv.org/abs/1912.04230}}}
	\end{center}
    }%
}
	
	\section{Introduction}\label{pf}
	We consider~$n$ nodes connected over a communication graph such that each node~$i$ has access to 
	a local cost function~$f_{i}:\mbb{R}^p\rightarrow\mbb{R}$. 
	The goal of the network is to solve the following optimization problem:
	\begin{align*}
	\label{eq:opt_problem}
	\mbox{P0}:
	\quad\min_{\mb{x}\in\mathbb{R}^p}f(\mb{x})\triangleq\frac{1}{n}\sum_{i=1}^{n}f_i(\mb{x}).
	\end{align*}
	Each node is only allowed to process its own local function and to exchange information with its neighboring nodes. 
	This formulation is well-known as decentralized optimization~\cite{DGD_tsitsiklis,DGD_nedich} that has been studied extensively by the control and signal processing communities over the past decade. Various decentralized approaches have been proposed, for example, Decentralized Gradient Descent (DGD)~\cite{DGD_nedich,DGD_Yuan,GP_neidch}, dual averaging~\cite{dual_averaging_Duchi,dual_averaging_Rabbat}, and ADMM~\cite{ADMM_Wei,ADMM_Shi}. More recently, significant effort has been made to design first-order gradient methods that achieve exact linear convergence for smooth and strongly-convex functions. Examples of such approaches include: primal methods, i.e., EXTRA~\cite{EXTRA}, Exact Diffusion~\cite{Exact_Diffusion}, and DLM~\cite{DLM}, methods based on gradient-tracking~\cite{GT_CDC,NEXT,harnessing,DIGing,add-opt, AGT,GT_jakovetic,Network-DANE} and~$\mc{AB}$/Push-Pull~\cite{AB, push-pull,TV-AB}; and dual methods, i.e.,~\cite{dual_optimal_ICML,dual_GT,dual_optimal_uribe}, that achieve better iteration complexity at the expense of computing the Fenchel dual gradient at each iteration. 
	
	In this paper, we focus on a refined formulation of decentralized optimization as follows:
	\begin{align*}
	\mbox{P1}:
	\quad\min_{\mb{x}\in\mathbb{R}^p}f(\mb{x})\triangleq\frac{1}{n}\sum_{i=1}^{n}f_i(\mb{x}),
	\qquad
	f_i(\mb{x}) \triangleq \frac{1}{m_i}\sum_{j=1}^{m_i} f_{i,j}(\mb{x}),
	\end{align*}	
	where we assume each local objective~$f_i$ is the average of several constituent functions~$\{f_{i,j}\}_{j=1}^{m_i}$. This formulation is motivated by large-scale data-science and machine learning, where large amount of training data is distributed over networked nodes (machines) and the goal is to train a model~$\mb{x}\in\mathbb{R}^p$ utilizing all local data. In Problem P1, each~$f_i = \frac{1}{m_i}\sum_{j=1}^{m_i}f_{i,j}$ is the local empirical risk function associated with the~$m_i$ training data samples at node~$i$. Towards Problem P1, various stochastic variants of DGD, EXTRA, Exact Diffusion and gradient tracking methods have been recently studied~\cite{DSGD_nedich,DGD_Kar,SGP_nedich,DSGD_NIPS,SGP_ICML,D2,DSGT_Pu,DSGT_Xin,SED,SGT_nonconvex_you}. These methods converge sub-linearly and outperform their deterministic counterparts when local data batches are large. 
	
	Finite-sum optimization problems have garnered a strong activity in the centralized settings and various variance-reduction techniques have been developed to accelerate the standard Stochastic Gradient Descent (SGD), for example, SAG~\cite{SAG}, SVRG~\cite{SVRG}, SAGA~\cite{SAGA}, Katyusha~\cite{katyusha}, SARAH~\cite{SARAH}, and several others. Such methods are shown to achieve fast linear convergence to the minimizer for smooth and strongly-convex functions, while maintaining comparable low per-iteration computation cost as SGD. It is therefore natural to introduce variance reduction to decentralized scenarios in order to improve the convergence and complexity aspects. In this paper, we borrow promising techniques from both centralized and decentralized settings, i.e., SAGA~\cite{SAGA} and stochastic gradient tracking methods~\cite{DSGT_Pu,DSGT_Xin}, and propose~\textbf{\texttt{GT-SAGA}}, a novel algorithm that achieves an accelerated linear convergence for smooth and strongly-convex functions.

	The convergence results of~\textbf{\texttt{GT-SAGA}} are based on the following assumptions. 
	\begin{assump}\label{asp1}
		Each local objective,~$f_{i,j}$, is~$\mu$-strongly-convex:~$\forall\mb{x}, \mb{y}\in\mbb{R}^p$, we have, for some~$\mu>0$, 
		\begin{equation*}
		f_{i,j}(\mb{y})\geq f_{i,j}(\mb{x})+ \big\langle\nabla f_{i,j}(\mb{x}), \mb{y}-\mb{x}\big\rangle+\frac{\mu}{2}\|\mb{x}-\mb{y}\|^2.
		\end{equation*}
	\end{assump}
	We note that under Assumption 1, the global objective function~$f$ has a unique minimizer, denoted as~$\mb{x}^*$.
	\begin{assump}\label{asp2}
		Each local objective,~$f_{i,j}$, is~$L$-smooth:~$\forall\mb{x}, \mb{y}\in\mbb{R}^p$, we have, for some~$L>0$,
		\begin{equation*}
		\qquad\|\mb{\nabla} f_{i,j}(\mb{x})-\mb{\nabla} f_{i,j}(\mb{y})\|\leq L\|\mb{x}-\mb{y}\|.
		\end{equation*}
	\end{assump}
	\begin{assump}\label{asp3}
		The weight matrix~$W$ associated with the graph,~$\mc{G}$, is primitive and doubly-stochastic.
	\end{assump}
	
	We denote~$\sigma$ as the second largest singular value of~$W$ and define~$M\triangleq\max_{i}m_i$,~$m\triangleq\min_{i}m_i$ and~$Q\triangleq L/\mu$, the condition number of~$f$. We show that \textbf{\texttt{GT-SAGA}} achieves~$\epsilon$-accuracy (in terms of distance to the minimizer)~with $$\mc{O}\left(\max\left\{M,\frac{M}{m}\frac{Q^2}{\left(1-\sigma\right)^2}\right\}\log\frac{1}{\epsilon}\right)$$
	local component gradient computations. Existing variance-reduced decentralized optimization methods include the following: DSA~\cite{DSA} that combines EXTRA~\cite{EXTRA} with SAGA~\cite{SAGA}; Diffusion-AVRG that combines Exact Diffusion~\cite{Exact_Diffusion} and AVRG~\cite{AVRG}; DSBA~\cite{DSBA} that adds proximal mapping~\cite{point-SAGA} to each iteration of DSA;~\cite{edge_DSA} that applies edge-based method~\cite{edge_AL} to DSA; ADFS~\cite{ADFS} that applies an accelerated  randomized proximal coordinate gradient method~\cite{APCG} to the dual formulation of Problem P1. We compare the convergence rate of~\textbf{\texttt{GT-SAGA}} with several state-of-the-art first-order primal methods that solve Problem P1 in Table 1, where, for the simplicity of presentation, we assume that all nodes have the same number of local functions, i.e., $M=m=\widetilde{m}$. It can be observed that in large-scale scenarios where~$\widetilde{m}$ is very large,~\textbf{\texttt{GT-SAGA}} improves upon the convergence rate of these methods in terms of the joint dependence on~$Q$ and~$\widetilde{m}$. We acknowledge that DSBA~\cite{DSBA} and ADFS~\cite{ADFS} achieve better iteration complexity than~\textbf{\texttt{GT-SAGA}}, however, at the expense of computing the proximal mapping of a component function at each iteration.
	Although the computation of this proximal mapping is efficient for certain function classes, it can be very expensive for general functions. Finally, it is worth noting that all existing variance-reduced decentralized stochastic methods~\cite{DSA,DAVRG,edge_DSA,DSBA,ADFS} require symmetric weight matrices and thus undirected networks. In contrast,~\textbf{\texttt{GT-SAGA}} only requires doubly-stochastic weights and therefore can be implemented over certain classes of directed graphs that admit doubly-stochastic weights~\cite{weight_balance_digraph}. This provides more flexibility in topology design of the network.
	
	We now describe the rest of the paper: Section~\ref{algorithm} formally describes the~\textbf{\texttt{GT-SAGA}} algorithm. Section~\ref{main proof} details the convergence analysis of the proposed algorithm.
	
	\begin{center}
		\begin{table}[]
			\caption{Comparison of several state-of-the-art decentralized optimization methods}
			\centering
			\begin{tabular}{|c|c|}
				\hline
				\textbf{Algorithm} & \textbf{Convergence Rate} \\ \hline
				Gradient Tracking~\cite{harnessing} & $\mc{O}\left(\frac{\widetilde{m}Q^2}{\left(1-\sigma\right)^2}\log\frac{1}{\epsilon}\right)$  \\ \hline
				Gradient Tracking with Nesterov acceleration (see Theorem 3 in~\cite{AGT}) & $\mc{O}\left(\frac{\widetilde{m}Q^{\frac{5}{7}}}{\sigma^{1.5}\left(1-\sigma\right)^{1.5}}\log\frac{1}{\epsilon}\right)$  \\ \hline
				DSA~\cite{DSA} & $\mc{O}\left(\max\left\{\widetilde{m}Q,\frac{Q^4}{1-\sigma},\frac{1}{(1-\sigma)^2}\right\}\log\frac{1}{\epsilon}\right)$\\ \hline
				Edge-based DSA~\cite{edge_DSA} & linear (no explicit rate provided in terms of~$\widetilde{m},Q,\sigma$) \\ \hline
				Diffusion-AVRG~\cite{DAVRG} & linear (no explicit rate provided in terms of~$\widetilde{m},Q,\sigma$) \\ \hline
				\textbf{\texttt{GT-SAGA} (this work)} &  $\mc{O}\left(\max\left\{\widetilde{m},\frac{Q^2}{\left(1-\sigma\right)^2}\right\}\log\frac{1}{\epsilon}\right)$ \\ \hline
			\end{tabular}
			\label{tab:my-table}
		\end{table}
	\end{center}
	
	\section{\textbf{\texttt{GT-SAGA}}: Algorithm Description}\label{algorithm}
	Towards Problem P1, we now formally introduce~\textbf{\texttt{GT-SAGA}} in Algorithm 1. As in stochastic gradient tracking methods~\cite{DSGT_Pu,DSGT_Xin}, each node~$i$ iteratively updates two vector variables~$\mb{x}_i^k$, the estimate of the minimizer~$\mb{x}^*$, and~$\mb{y}_i^k$, the local gradient tracker. We note that~$\mb{z}_{i,j}^k$ is an auxiliary variable maintained at each node~$i$ that denotes the most recent point where the gradient of the component function~$f_{i,j}$ was computed before time~$k$ and is not explicitly used in the practical implementation.
	Intuitively, the local SAGA gradient~$\mb{g}_i^{k}$ is an unbiased estimator of the local full gradient~$\nabla f_i(\mb{x}_i^k)$ with decreasing variance as~$\mb{x}_i^k$ approaches to~$\mb{x}^*$. The average (over the nodes) of the local gradient tracker~$\mb{y}_i^k$ deterministically preserves the average of all local SAGA gradients,~$\frac{1}{n}\sum_{i=1}^{n}\mb{g}_i^k$, and therefore asymptotically approach to the gradient of the global objective function. In the rest of the paper, we assume~$p=1$ for the sake of simplicity. It is straightforward to develop the general case of~$p>1$ with the help of the Kronecker products; see e.g., the procedure in~\cite{AB}.
	\begin{algorithm}[H]
		\caption{~\textbf{\texttt{GT-SAGA}} at each node~$i$}
		\begin{algorithmic}[1]
			\Require Arbitrary starting point~$\mb{x}_i^0\in\mbb{R}^p$ and~step-size~$\alpha>0$. 
			
			~~Local gradient table:~$\{\nabla f_{i,j}(\mb{z}_{i,j}^0)\}_{j=1}^{m_i}$ with~$\mb{x}_i^0=\mb{z}_{i,j}^0=\mb{z}_{i,j}^1$,~$\forall j$. 
			
			~~Gradient tracker:~$\mb{y}_i^0=\mb{g}_i^0=\frac{1}{m_i}\sum_{j=1}^{m_i}\nabla f_{i,j}(\mb{z}_{i,j}^{0})$. 
			
			~~Doubly stochastic weights:~$W=\{w_{ir}\}\in\mathbb{R}^{n\times n}$.
			\For{$k= 0,1,2,\cdots$}
			\State $\mb{x}_{i}^{k+1} = \sum_{r=1}^{n}w_{ir}\mb{x}_{r}^{k} - \alpha\mb{y}_{i}^{k}$ \Comment{Estimate update}
			\State Select~$s_{i}^{k+1}$ uniformly at random from~$\{1,\cdots,m_i\}$.
			\Comment{Sample from  local data}
			\State $\mb{g}_{i}^{k+1} = \nabla f_{i,s_i^{k+1}}(\mb{x}_{i}^{k+1}) - \nabla f_{i,s_i^{k+1}}(\mb{z}_{i,s_i^{k+1}}^{k+1}) + \frac{1}{m_i}\sum_{j=1}^{m_i}\nabla f_{i,j}(\mb{z}_{i,j}^{k+1})$ \Comment{Local SAGA update}
			\State $\mb{y}_{i}^{k+1} = \sum_{r=1}^{n}w_{ir}\mb{y}_{r}^{k} + \mb{g}_i^{k+1} - \mb{g}_i^{k}$ \Comment{Gradient Tracker update}
			\State Replace~$\nabla f_{i,s_i^{k+1}}(\mb{z}_{i,s_i^{k+1}}^{k+1})$ by~$\nabla f_{i,s_i^{k+1}}(\mb{x}_{i}^{k+1})$ in the local gradient table
			\Comment{Update local gradient table}
			\If{$j = s_{i}^{k+1}$}
			$\mb{z}_{i,j}^{k+2} = \mb{x}_i^{k+1}$
			\Else
			~$\mb{z}_{i,j}^{k+2} = \mb{z}_{i,j}^{k+1}$
			\EndIf
			\EndFor
		\end{algorithmic}
	\end{algorithm}
	
	\section{\textbf{\texttt{GT-SAGA}}: Convergence Analysis}\label{main proof}
	\subsection{Preliminaries}
	The randomness of~\textbf{\texttt{GT-SAGA}} lies in the set of independent random variables~$\{s_i^k\}_{i\in\mc{V}}^{k\geq1}$. We denote~$\mc{F}^k$ as the~$\sigma$-algebra generated by~$\{s_i^t\}^{t\leq k-1}_{i\in\mc{V}}$.
	We note that~$\{\mb{x}_i^t\}^{t\leq k}_{i\in\mc{V}}$,~$\{\mb{z}_{i,j}^t\}^{t\leq k}_{i\in\mc{V}}$,~$\{\mb{g}_i^t\}^{t\leq k-1}_{i\in\mc{V}}$ and~$\{\mb{y}_i^t\}^{t\leq k-1}_{i\in\mc{V}}$ are fixed given~$\mc{F}^k$ and~$\mathbb{E}\left[\:\cdot\:|\mc{F}^k\right]$ denotes the conditional expectation over~$\{s_i^k\}_{i\in\mc{V}}$ given~$\mc{F}^k$. We now write~\textbf{\texttt{GT-SAGA}} in the following compact matrix form for the sake of analysis:
	\begin{subequations}\label{WWSAGA}
		\begin{align}
		\mb{x}^{k+1} &= W\mb{x}^k - \alpha\mb{y}^k, \label{av}\\
		\mb{y}^{k+1} &= W\mb{y}^k +
		\mb{g}^{k+1}-\mb{g}^k, \label{bv}
		\end{align}
	\end{subequations}
	where we use the following notation:
	\begin{align*}
	\mb{x}^k\triangleq\left[{\mb{x}^{k}_{1}}^\top,\cdots,{\mb{x}^{k}_{n}}^\top\right]^\top, \quad
	\mb{y}^k\triangleq\left[{\mb{y}^{k}_{1}}^\top,\cdots,{\mb{y}^{k}_{n}}^\top\right]^\top, \quad
	\mb{g}^k\triangleq\left[{\mb{g}^{k}_{1}}^\top,\cdots,{\mb{g}^{k}_{n}}^\top\right]^\top.
	\end{align*}
	We also define the following quantities:
	$$\ol{\mb{x}}^k \triangleq \frac{1}{n}\mb{1}_n^\top\mb{x}^k,~~\ol{\mb{y}}^k \triangleq \frac{1}{n}\mb{1}_n^\top\mb{y}^k,~~ \ol{\mb{g}}^k\triangleq \frac{1}{n}\mb{1}_n^\top\mb{g}^k,~~\nabla\mb{f}(\mb{x}^k)\triangleq[\nabla f_1(\mb{x}_1^k)^\top,\dots,\nabla f_n(\mb{x}_n^k)^\top]^\top,~~\mb{h}(\mb{x}^k)\triangleq\frac{1}{n}\mb{1}_n^\top\nabla\mb{f}(\mb{x}^k).$$
	The Lemmas in this subsection are standard in the literature of stochastic gradient tracking methods and SAGA. Their proofs can be found in, for example,~\cite{harnessing,DSGT_Pu,DSGT_Xin,SAGA,DIGing}.

	Each local SAGA gradient~$\mb{g}_i^k$ is an unbiased estimator of the local full gradient~$\nabla f_i({\mb{x}_i^k})$.
	\begin{lem}\label{p1}
		$\mathbb{E}\left[\mb{g}^k|\mc{F}^k\right] = \nabla \mb{f}(\mb{x}^k),~\forall k\geq 0$.
	\end{lem}
	The average of gradient trackers~$\{\mb{y}_i^k\}$ preserves the average of local SAGA gradients~$\{\mb{g}_i^k\}$.
	\begin{lem}\label{p2}
		$\ol{\mb{y}}^k = \ol{\mb{g}}^k, \forall k\geq0$.
	\end{lem}
	Based on Lemma~\ref{p1} and~\ref{p2}, the following is straightforward.
	\begin{lem}\label{p3}
		$\mathbb{E}\left[\ol{\mb{y}}^k|\mc{F}^k\right] = \mb{h}(\mb{x}^k)$,~$\forall k\geq0$.
	\end{lem}
	The difference of~$\mb{h}(\mb{x}^k)$ and~$\nabla f(\ol{\mb{x}}^k)$ is bounded by the consensus error~$\left\|\mb{x}^k-\mb{1}_n\ol{\mb{x}}^k\right\|$ as follows. 
	\begin{lem}\label{p4}
		$\left\|\mb{h}(\mb{x}^k)-\nabla f(\ol{\mb{x}}^k)\right\|\leq\frac{L}{\sqrt{n}}\left\|\mb{x}^k-\mb{1}_n\ol{\mb{x}}^k\right\|$,~$\forall k\geq0$.
	\end{lem}
	The weight matrix~$W$ is a contraction operator.
	\begin{lem}
		$\forall \mb{x}\in\mathbb{R}^n, \left\|W\mb{x} - W_\infty\mb{x}\right\|\leq\sigma\left\|\mb{x} - W_\infty\mb{x}\right\|$, where~$W_\infty=\frac{1}{n}\mb{1}_n\mb{1}_n^\top$.
	\end{lem}
	Descending along the direction of full gradient leads to a contraction in the optimality gap~\cite{nesterov_book}.
	\begin{lem}
		Let~$f$ be~$\mu$-strongly-convex and~$L$-smooth. If~$0<\alpha\leq\frac{1}{L}$, the following holds, for~$\forall\mb{x}\in\mathbb{R}^p$,
		\begin{align*}
		\left\|\mb{x}-\alpha\nabla f(\mb{x}) -\mb{x}^*\right\|
		\leq(1-\mu\alpha)\left\|\mb{x}-\mb{x}^*\right\|
		\end{align*}
	\end{lem}
	
	With the help of these Lemmas, we now proceed with the convergence analysis of~\textbf{\texttt{GT-SAGA}}.
	\subsection{Auxiliary Results}
	Following~\cite{DSGT_Pu,DSGT_Xin}, we first derive a contraction + perturbation bound for the consensus error~$\left\|\mb{x}^{k}-\mb{1}_n\ol{\mb{x}}^{k}\right\|^2$.
	\begin{lem}\label{saga_1}
		$\forall k \geq0$,
		$\mathbb{E}\left[\left\|\mb{x}^{k+1}-\mb{1}_n\ol{\mb{x}}^{k+1}\right\|^2|\mc{F}^k\right]\leq
		\frac{1+\sigma^2}{2}\left\|\mb{x}^k-\mb{1}_n\ol{\mb{x}}^k\right\|^2 + \frac{2\alpha^2}{1-\sigma^2}\mathbb{E}\left[\left\|\mb{y}^k-\mb{1}_n\ol{\mb{y}}^k\right\|^2|\mc{F}^k\right].$
	\end{lem}
	\begin{proof}
		Following from~\eqref{av}, we have
		\begin{align*}
		&\left\|\mb{x}^{k+1}-\mb{1}_n\ol{\mb{x}}^{k+1}\right\|^2 = \left\|W\mb{x}^k - \alpha\mb{y}^k-W_\infty\left(W\mb{x}^k - \alpha\mb{y}^k\right)\right\|^2 \\
		=& \left\|W\mb{x}^k - W_\infty\mb{x}^k\right\|^2 + \alpha^2\left\|\mb{y}^k - W_\infty\mb{y}^k\right\|^2
		- 2\alpha\Big\langle W\mb{x}^k -W_\infty\mb{x}^k, \mb{y}^k -W_\infty\mb{y}^k \Big\rangle \\
		=&~\sigma^2\left\|\mb{x}^k - W_\infty\mb{x}^k\right\|^2 + \alpha^2\left\|\mb{y}^k - W_\infty\mb{y}^k\right\|^2
		+2\sigma\left\|\mb{x}^k - W_\infty\mb{x}^k\right\|\alpha\left\|\mb{y}^k - W_\infty\mb{y}^k\right\| \\
		\leq&~\sigma^2\left\|\mb{x}^k - W_\infty\mb{x}^k\right\|^2 + \alpha^2\left\|\mb{y}^k - W_\infty\mb{y}^k\right\|^2 + \sigma\left(\frac{1-\sigma^2}{2\sigma}\left\|\mb{x}^k - W_\infty\mb{x}^k\right\|^2
		+ \frac{2\sigma}{1-\sigma^2}\alpha^2\left\|\mb{y}^k - W_\infty\mb{y}^k\right\|^2\right),\\
		=&~\frac{1+\sigma^2}{2}\left\|\mb{x}^k - W_\infty\mb{x}^k\right\|^2+\alpha^2\left(1+\frac{2\sigma^2}{1-\sigma^2}\right)\left\|\mb{y}^k - W_\infty\mb{y}^k\right\|^2
		\end{align*}
		and the proof follows from~$1+\sigma^2<2$ and taking the conditional expectation given~$\mc{F}^k$.
	\end{proof}
	
	The next Lemma derives a contraction + perturbation bound for the optimality gap of the variables~$\mb{z}_{i,j}^k$. 
	\begin{lem}
		We define~$\mb{t}_i^k$ and~$\mb{t}^k$ as follows:
		\begin{align*}
		\mb{t}_i^k \triangleq \frac{1}{m_i}\sum_{j=1}^{m_i}\left\|\mb{z}_{i,j}^k-\mb{x}^*\right\|^2,\qquad\mb{t}^k \triangleq \frac{1}{n}\sum_{i=1}^{n}\mb{t}_i^k.
		\end{align*}
		We define~$M\triangleq\max\{m_i\}$ and~$m\triangleq\min\{m_i\}$. Then the following holds:
		\begin{align*}
		\mathbb{E}\left[\mb{t}^{k+1}|\mc{F}^k\right]\leq
		\left(1-\frac{1}{M}\right)\mb{t}^k + \frac{2}{mn}\left\|\mb{x}^k-\mb{1}_n\ol{\mb{x}}^k\right\|^2
		+ \frac{2}{m}\left\|\ol{\mb{x}}^k-\mb{x}^*\right\|^2,
		\qquad\forall k\geq 0.
		\end{align*}
	\end{lem}
	\begin{proof}
		We note that~$\mb{z}_{i,j}^{k+1} = \mb{z}_{i,j}^{k}$ with probability~$1-\frac{1}{m_i}$ and~$\mb{z}_{i,j}^{k+1} = \mb{x}_{i}^{k}$ with probability~$\frac{1}{m_i}$, given~$\mc{F}^k$.
		\begin{align*}
		\mathbb{E}\left[\mb{t}_i^{k+1}|\mc{F}^k\right]
		=&~ \frac{1}{m_i}\sum_{j=1}^{m_i}\mathbb{E}\left[\left\|\mb{z}_{i,j}^{k+1}-\mb{x}^*\right\|^2|\mc{F}^k\right]\nonumber\\
		=&~ \frac{1}{m_i}\sum_{j=1}^{m_i}\left(\left(1-\frac{1}{m_i}\right)\left\|\mb{z}_{i,j}^{k}-\mb{x}^*\right\|^2+\frac{1}{m_i}\left\|\mb{x}_i^k-\mb{x}^*\right\|^2\right)\\
		=&~
		\left(1-\frac{1}{m_i}\right)\mb{t}_i^k + \frac{1}{m_i}\left\|\mb{x}_i^k-\mb{x}^*\right\|^2
		\nonumber\\
		\leq&~
		\left(1-\frac{1}{M}\right)\mb{t}_i^k + \frac{1}{m}\left\|\mb{x}_i^k-\mb{x}^*\right\|^2\\
		\leq&~
		\left(1-\frac{1}{M}\right)\mb{t}_i^k + \frac{2}{m}\left\|\mb{x}_i^k-\ol{\mb{x}}^k\right\|^2
		+ \frac{2}{m}\left\|\ol{\mb{x}}^k-\mb{x}^*\right\|^2
		\end{align*}
		Averaging the above over~$i$ finishes the proof.
	\end{proof}
	
	The next Lemma provides a contraction + consensus perturbation + variance bound for~$\left\|\ol{\mb x}^{k+1}-\mb{x}^*\right\|^2$.
	\begin{lem}\label{x-x*}
		Bound the optimality gap as follows.
		\begin{align*}
		&\mathbb{E}\left[\left\|\ol{\mb x}^{k+1}-\mb{x}^*\right\|^2 | \mc{F}_{k}\right] \nonumber\\
		=& \left\|\ol{\mb x}^{k}-\alpha\nabla f(\ol{\mb x}^{k})-\mb{x}^* \right\|^2 + 2\alpha\Big\langle \ol{\mb x}^{k}-\alpha\nabla f(\ol{\mb x}^{k})-\mb{x}^*, \nabla f(\ol{\mb x}^{k})-\mb{h}(\mb{x}^k)\Big\rangle+\alpha^2\left\|\nabla f(\ol{\mb x}^{k})
		-\mb{h}(\mb{x}^k)\right\|^2. \nonumber\\
		&+\frac{\alpha^2}{n^2}\mathbb{E}\left[\left\| \mb{g}^k-\nabla\mb{f}(\mb{x}^k)\right\|^2 
		\Big| \mc{F}^k\right].
		\end{align*}
	\end{lem}
	\begin{proof}
		Multiplying~$\frac{1}{n}\mb{1}_n^\top$ to bothsides of~\eqref{av}, we have~$\ol{\mb x}^{k+1}=\ol{\mb x}^{k}-\alpha\ol{\mb y}^{k}$. We next expand~$\left\|\ol{\mb x}^{k+1}-\mb{x}^*\right\|^2$.
		\begin{align*}
		&\left\|\ol{\mb x}^{k+1}-\mb{x}^*\right\|^2 = \left\|\ol{\mb x}^{k}-\alpha\ol{\mb y}^{k}-\mb{x}^*\right\|^2\\
		=& \left\|\ol{\mb x}^{k}-\alpha\nabla f(\ol{\mb x}^{k})-\mb{x}^*+\alpha\left(\nabla f(\ol{\mb x}^{k})-\ol{\mb y}^{k}\right)\right\|^2 \\
		=& \left\|\ol{\mb x}^{k}-\alpha\nabla f(\ol{\mb x}^{k})-\mb{x}^* \right\|^2 + 2\alpha\Big\langle \ol{\mb x}^{k}-\alpha\nabla f(\ol{\mb x}^{k})-\mb{x}^*, \nabla f(\ol{\mb x}^{k})-\ol{\mb y}^{k}\Big\rangle + \alpha^2 \left\|\nabla f(\ol{\mb x}^{k})-\ol{\mb y}^{k}\right\|^2. 
		\end{align*}
		Recall that~$\mathbb{E}\left[\ol{\mb{y}}^k|\mc{F}_k\right] = \mb{h}(\mb{x}^k)$ from Lemma~\ref{p3}. We take the expectation from bothsides given~$\mc{F}^k$ to obtain:
		\begin{align}\label{d1}
		\mathbb{E}\left[\left\|\ol{\mb x}^{k+1}-\mb{x}^*\right\|^2 | \mc{F}^k\right] 
		=& \left\|\ol{\mb x}^{k}-\alpha\nabla f(\ol{\mb x}^{k})-\mb{x}^* \right\|^2 + 2\alpha\Big\langle \ol{\mb x}^{k}-\alpha\nabla f(\ol{\mb x}^{k})-\mb{x}^*, \nabla f(\ol{\mb x}^{k})-\mb{h}(\mb{x}^k)\Big\rangle
		\nonumber\\
		&+ \alpha^2 \mathbb{E}\left[\left\|\nabla f(\ol{\mb x}^{k})-\ol{\mb y}^{k}\right\|^2 | \mc{F}^k\right].
		\end{align}
		We split the last term above~$\left\|\nabla f(\ol{\mb x}^{k})-\ol{\mb y}^{k}\right\|^2$ as consensus error + variance as follows. 
		\begin{align}\label{e1}
		&\mathbb{E}\left[\left\|\nabla f(\ol{\mb x}^{k})-\ol{\mb y}^{k}\right\|^2 | \mc{F}^{k}\right]\nonumber\\
		=& \mathbb{E}\left[\left\|\nabla f(\ol{\mb x}^{k})
		-\mb{h}(\mb{x}^k) + \mb{h}(\mb{x}^k)
		-\ol{\mb y}^{k}\right\|^2 | \mc{F}^{k}\right] \nonumber\\
		=& \underbrace{\left\|\nabla f(\ol{\mb x}^{k})
			-\mb{h}(\mb{x}^k)\right\|^2 }_{\text{consensus error}} 
		+ \underbrace{\mathbb{E}\left[\left\|\mb{h}(\mb{x}^k)
			-\ol{\mb y}^{k}\right\|^2 \big| \mc{F}^{k}\right]}_{\text{variance}}
		+ \underbrace{2\Big\langle \nabla f(\ol{\mb x}^{k})
			-\mb{h}(\mb{x}^k), \mathbb{E}\left[\mb{h}(\mb{x}^k)
			-\ol{\mb y}^{k}\big| \mc{F}^k \right]\Big\rangle}_{=0} 
		\end{align}
		The variance term can be simplified as follows:
		\begin{align}\label{v1}
		&~\mathbb{E}\left[\left\|\mb{h}(\mb{x}^k)
		-\ol{\mb y}^{k}\right\|^2 \Big| \mc{F}^{k}\right] \nonumber\\
		=&~\mathbb{E}\left[\left\|
		\frac{1}{n}\sum_{i=1}^{n}\left( \nabla f_i(\mb{x}_i^k)-\mb{g}_i^k\right)
		\right\|^2 \Bigg| \mc{F}^{k}\right] 
		= \frac{1}{n^2}\mathbb{E}\left[\left\|
		\sum_{i=1}^{n}\left(\nabla f_i(\mb{x}_i^k)-\mb{g}_i^k\right)
		\right\|^2 \Bigg| \mc{F}^{k}\right] \nonumber\\
		=&~\frac{1}{n^2}\mathbb{E}\left[\sum_{i=1}^{n}\left\|
		f_i'(\mb{x}_i^k)-\mb{g}_i^k\right\|^2 
		+\sum_{i\neq j}\Big\langle f_i'(\mb{x}_i^k)-\mb{g}_i^k, f_j'(\mb{x}_j^k)-\mb{g}_j^k\Big\rangle 
		\Bigg| \mc{F}^{k}\right] \nonumber\\
		=&~\frac{1}{n^2}\sum_{i=1}^{n}\mathbb{E}\left[\left\|
		f_i'(\mb{x}_i^k)-\mb{g}_i^k\right\|^2 
		\Big| \mc{F}^{k}\right]
		=\frac{1}{n^2}\mathbb{E}\left[\left\|\mb{g}^k-\nabla\mb{f}(\mb{x}^k)\right\|^2\Big|\mc{F}^k\right]
		,	
		\end{align}                                                   
		where the second last equality is due to the fact that~$\{\mb{g}_i^k\}$ are independent with each other given~$\mc{F}^k$. Using~\eqref{v1} and~\eqref{e1} in~\eqref{d1} finishes the proof.
	\end{proof}
	
	Following a similar procedure in SAGA~\cite{SAGA}, we bound the variance~$\mathbb{E}\left[\left\| \mb{g}_i^k-\nabla f_i(\mb{x}_i^k)\right\|^2 
	\Big| \mc{F}^{k}\right]$ as follows.
	\begin{lem}\label{v3}
		The following holds:
		\begin{align}
		\mathbb{E}\left[\left\| \mb{g}_i^k-\nabla f_i(\mb{x}_i^k)\right\|^2 
		\Big| \mc{F}^{k}\right] 
		\leq&
		\left(1+\beta\right)\frac{1}{m_i}\sum_{j=1}^{m_i}\left\| \nabla f_{i,j}(\mb{x}_{i}^{k})
		- \nabla f_{i,j}(\mb{x}^*)\right\|^2 
		-\left(1+\beta\right)\left\| \nabla f_i(\mb{x}_i^k)-\nabla f_i(\mb{x}^*)\right\|^2  \nonumber\\
		&+ \left(1+\beta^{-1}\right)\frac{1}{m_i}\sum_{j=1}^{m_i}\left\| \nabla f_{i,j}(\mb{z}_{i,j}^{k}) - \nabla f_{i,j}(\mb{x}^*)\right\|^2
		\quad\forall k\geq0.
		\end{align}
	\end{lem}
	\begin{proof} 
		We define~$\mb{d}_i^k\triangleq\frac{1}{m_i}\sum_{j=1}^{m_i}\nabla f_{i,j}(\mb{z}_{i,j}^{k})$. The key is to use the standard variance decomposition.  
		\begin{align}\label{v2}
		&\mathbb{E}\left[\left\|\mb{g}_i^k-\nabla f_i(\mb{x}_i^k)\right\|^2 
		| \mc{F}^{k}\right] \nonumber\\
		=&~\mathbb{E}\left[\left\|
		\nabla f_{i,s_i^{k}}(\mb{x}_{i}^{k}) - \nabla f_{i,s_i^{k}}(\mb{z}_{i,s_i^{k}}^{k}) + \mb{d}_i^k-\nabla f_i(\mb{x}_i^k)\right\|^2 
		\Bigg| \mc{F}^{k}\right] \nonumber\\
		=&~\mathbb{E}\left[\left\|
		\nabla f_{i,s_i^{k}}(\mb{x}_{i}^{k})-\nabla f_i(\mb{x}_i^k) - \left(\nabla f_{i,s_i^{k}}(\mb{z}_{i,s_i^{k}}^{k}) - \mb{d}_i^k\right)\right\|^2  
		\Bigg| \mc{F}^{k}\right] \nonumber\\
		=&~\mathbb{E}\Bigg[\Bigg\|
		\Bigg(\underbrace{\nabla f_{i,s_i^{k}}(\mb{x}_{i}^{k})
			- \nabla f_{i,s_i^{k}}(\mb{x}^*)}_{X_i^k}
		-\underbrace{\left( \nabla f_i(\mb{x}_i^k)-\nabla f_i(\mb{x}^*)\right)}_{\mathbb{E}X_i^k}\Bigg) \nonumber\\
		&\qquad\qquad\qquad-
		\Bigg(\underbrace{ \nabla f_{i,s_i^{k}}(\mb{z}_{i,s_i^{k}}^{k}) - \nabla f_{i,s_i^{k}}(\mb{x}^*)}_{Y_i^k}- \underbrace{\left(\mb{d}_i^k-\nabla f_i(\mb{x}^*)\right)}_{\mathbb{E}Y_i^k}\Bigg)\Bigg\|^2  
		\Bigg| \mc{F}^{k}\Bigg] \nonumber
		\end{align}
		We use the inequality~$\|\mb{x}+\mb{y}\|^2\leq(1+\beta)\|\mb{x}\|^2+(1+\beta^{-1})\|\mb{y}\|^2, \forall\beta>0$, and the standard variance decomposition~$\mathbb{E}\|X-\mathbb{E}X\|^2 = \mathbb{E}\|X\|^2 - \|\mathbb{E}X\|^2$ to proceed.
		\begin{align}
		&\mathbb{E}\left[\left\|\mb{g}_i^k-\nabla f_i(\mb{x}_i^k)\right\|^2 
		| \mc{F}_{k}\right] \nonumber\\
		\leq&
		\left(1+\beta\right)
		\mathbb{E}\left[\left\| f_{i,s_i^{k}}(\mb{x}_{i}^{k})
		- \nabla f_{i,s_i^{k}}(\mb{x}^*)
		-\left( \nabla f_i(\mb{x}_i^k)-\nabla f_i(\mb{x}^*)\right)\right\|^2 \Bigg | \mc{F}^k \right]
		\nonumber\\
		&\qquad\qquad+ \left(1+\beta^{-1}\right)\mathbb{E}\left[\left\| 
		\nabla f_{i,s_i^{k}}(\mb{z}_{i,s_i^{k}}^{k}) - \nabla f_{i,s_i^{k}}(\mb{x}^*) - \left(\mb{d}_i^k-\nabla f_i(\mb{x}^*)\right)
		\right\|^2 \Bigg |\mc{F}^k\right] \nonumber\\
		=& 
		\left(1+\beta\right)\left(\mathbb{E}\left[\left\| \nabla f_{i,s_i^{k}}(\mb{x}_{i}^{k})
		- \nabla f_{i,s_i^{k}}(\mb{x}^*)\right\|^2\Big|\mc{F}_k\right] -\left\| \nabla f_i(\mb{x}_i^k) - \nabla f_i(\mb{x}^*)\right\|^2 \right) \nonumber\\
		&\qquad\qquad+ \left(1+\beta^{-1}\right)\left(\mathbb{E}\left[\left\| \nabla f_{i,s_i^{k}}(\mb{z}_{i,s_i^{k}}^{k}) - \nabla f_{i,s_i^{k}}(\mb{x}^*)\right\|^2\Big|\mc{F}_k\right] -\left\| \mb{d}_i^k-\nabla f_i(\mb{x}^*)\right\|^2 \right)\nonumber\\
		\leq&
		\left(1+\beta\right)\frac{1}{m_i}\sum_{j=1}^{m_i}\left\| \nabla f_{i,j}(\mb{x}_{i}^{k})
		- \nabla f_{i,j}(\mb{x}^*)\right\|^2 
		-\left(1+\beta\right)\left\| \nabla f_i(\mb{x}_i^k)-\nabla f_i(\mb{x}^*)\right\|^2  \nonumber\\
		&+ \left(1+\beta^{-1}\right)\frac{1}{m_i}\sum_{j=1}^{m_i}\left\|\nabla f_{i,j}(\mb{z}_{i,j}^{k}) - \nabla f_{i,j}(\mb{x}^*)\right\|^2,
		\end{align}
		where in the last inequality we dropped the non-positive term~$-\left\|\mb{d}_i^k - f'_i(\mb{x}^*)\right\|^2$.
	\end{proof}
	Next we use the lemma above in a slightly conservative way (other ways of doing it are possible). We simply set~$\beta=1$ and drop the negative term above.
	\begin{align}\label{v'}
	\mathbb{E}\left[\left\| \mb{g}_i^k-\nabla f_i(\mb{x}_i^k)\right\|^2 
	\Big| \mc{F}^{k}\right] 
	\leq&~
	\frac{2}{m_i}\sum_{j=1}^{m_i}\left\|\nabla f_{i,j}(\mb{x}_{i}^{k})
	- \nabla f_{i,j}(\mb{x}^*)\right\|^2 
	+ \frac{2}{m_i}\sum_{j=1}^{m_i}\left\| \nabla f_{i,j}(\mb{z}_{i,j}^{k}) - \nabla f_{i,j}(\mb{x}^*)\right\|^2 \nonumber\\
	\triangleq&~I_1 + I_2.
	\end{align}
	First we bound~$I_1$. We add and subtract~$\nabla f_{i,j}(\ol{\mb{x}}^k)$:
	\begin{align}\label{I1}
	I_1 =&~\frac{2}{m_i}\sum_{j=1}^{m_i}\left\| \nabla f_{i,j}(\mb{x}_{i}^{k}) - \nabla f_{i,j}(\ol{\mb{x}}^k)
	+ \nabla f_{i,j}(\ol{\mb{x}}^k)
	- \nabla f_{i,j}(\mb{x}^*)\right\|^2 \nonumber\\
	\leq&~
	\frac{4}{m_i}\sum_{j=1}^{m_i}
	\left(
	L^2\left\|\mb{x}_{i}^{k} - \ol{\mb{x}}^k\right\|^2
	+ L^2\left\|\ol{\mb{x}}^k - \mb{x}^*\right\|^2   
	\right) \nonumber\\
	\leq&~
	4L^2\left\|\mb{x}_i^k - \ol{\mb{x}}^k\right\|^2
	+ 4L^2\left\|\ol{\mb{x}}^k - \mb{x}^*\right\|^2
	\end{align}
	Next we bound~$I_2$ as follows.
	\begin{align}\label{I2}
	I_2 = \frac{2}{m_i}\sum_{j=1}^{m_i}\left\| f_{i,j}'(\mb{z}_{i,j}^{k}) - f_{i,j}'(\mb{x}^*)\right\|^2
	\leq \frac{2L^2}{m_i}\sum_{j=1}^{m_i}\left\| \mb{z}_{i,j}^{k} - \mb{x}^*\right\|^2
	= 2L^2\mb{t}_i^k.
	\end{align}
	Using the bounds~\eqref{I1} and~\eqref{I2} in~\eqref{v'}, we obtain an upper bound for the local variance:
	$$\mathbb{E}\left[\left\| \mb{g}_i^k-\nabla f_i(\mb{x}_i^k)\right\|^2 
	\Big| \mc{F}^{k}\right] 
	\leq 4L^2\left\|\mb{x}_i^k - \ol{\mb{x}}^k\right\|^2
	+ 4L^2\left\|\ol{\mb{x}}^k - \mb{x}^*\right\|^2
	+ 2L^2\mb{t}_i^k.
	$$
	Summing the above inequality over~$i$, we obtain the following Lemma.
	\begin{lem}\label{var}
		The following holds:
		\begin{align*}
		\mathbb{E}\left[\left\| \mb{g}^k-\nabla \mb{f}(\mb{x}^k)\right\|^2 
		\Big| \mc{F}^{k}\right] 
		\leq
		4L^2\left\|\mb{x}^k-\mb{1}_n\ol{\mb{x}}^k\right\|^2
		+ 4nL^2\left\|\ol{\mb{x}}^k - \mb{x}^*\right\|^2
		+ 4nL^2\mb{t}^k,\qquad\forall k\geq 0.
		\end{align*}
	\end{lem}
	Now we further refine Lemma~\ref{x-x*} with the help of Lemma~\ref{var}.
	\begin{lem}\label{opt_saga}
		If~$0<\alpha \leq \frac{\mu}{8L^2}$, the following holds:
		\begin{align*}
		\mathbb{E}\left[\left\|\ol{\mb x}^{k+1}-\mb{x}^*\right\|^2 | \mc{F}^{k}\right]
		\leq
		\left(1-\frac{\mu\alpha}{2}\right)\left\|\ol{\mb x}^{k}-\mb{x}^*\right\|^2
		+ \frac{3L^2\alpha}{2\mu n}
		\left\|\mb{x}^k-\mb{1}_n\ol{\mb{x}}^k\right\|^2
		+ \frac{4L^2\alpha^2}{n}\mb{t}^k,\quad\forall k\geq 0.
		\end{align*}
	\end{lem}
	\begin{proof}
		Recall Lemma~\ref{x-x*} and use standard contraction in gradient descent. 
		\begin{align*}
		&\mathbb{E}\left[\left\|\ol{\mb x}^{k+1}-\mb{x}^*\right\|^2 | \mc{F}^{k}\right] \nonumber\\
		\leq&~
		\left(1-\mu\alpha\right)^2\left\|\ol{\mb x}^{k}-\mb{x}^*\right\|^2
		+ 2\alpha\left(1-\mu\alpha\right)\left\|\ol{\mb x}^{k}-\mb{x}^*\right\|\left\|\nabla f(\ol{\mb x}^{k})
		-\mb{h}(\mb{x}^k)\right\|+\alpha^2\left\|\nabla f(\ol{\mb x}^{k})
		-\mb{h}(\mb{x}^k)\right\|^2 \nonumber\\
		&+\frac{\alpha^2}{n^2}\mathbb{E}\left[\left\| \mb{g}^k-\nabla\mb{f}(\mb{x}^k)\right\|^2 
		\Big| \mc{F}^{k}\right] \nonumber\\
		\leq&~
		\left(1-\mu\alpha\right)^2\left\|\ol{\mb x}^{k}-\mb{x}^*\right\|^2
		+ \alpha\left(1-\mu\alpha\right)\left(\mu\left\|\ol{\mb x}^{k}-\mb{x}^*\right\|^2+\frac{1}{\mu}\left\|\nabla f(\ol{\mb x}^{k})
		-\mb{h}(\mb{x}^k)\right\|^2\right) \nonumber\\
		&+\alpha^2\left\|\nabla f(\ol{\mb x}^{k})
		-\mb{h}(\mb{x}^k)\right\|^2 +\frac{\alpha^2}{n^2}\mathbb{E}\left[\left\| \mb{g}^k-\nabla\mb{f}(\mb{x}^k)\right\|^2 
		\Big| \mc{F}^{k}\right] \nonumber\\
		=&~
		\left(1-\mu\alpha\right)\left\|\ol{\mb x}^{k}-\mb{x}^*\right\|^2
		+\frac{\alpha}{\mu}\left\|\nabla f(\ol{\mb x}^{k})
		-\mb{h}(\mb{x}^k)\right\|^2 +\frac{\alpha^2}{n^2}\mathbb{E}\left[\left\| \mb{g}^k-\nabla\mb{f}(\mb{x}^k)\right\|^2 
		\Big| \mc{F}^{k}\right]. \nonumber
		\end{align*}
		Applying Lemma~\ref{p4} and~\ref{var} to the inequality above, we have:
		\begin{align*}
		&\mathbb{E}\left[\left\|\ol{\mb x}^{k+1}-\mb{x}^*\right\|^2 | \mc{F}^{k}\right] \nonumber\\
		\leq&~
		\left(1-\mu\alpha\right)\left\|\ol{\mb x}^{k}-\mb{x}^*\right\|^2 + \frac{\alpha L^2}{\mu n}\left\|\mb{x}^k-\mb{1}_n\ol{\mb{x}}^k\right\|^2
		+\frac{\alpha^2}{n^2}\left(4L^2\left\|\mb{x}^k-\mb{1}_n\ol{\mb{x}}^k\right\|^2
		+ 4nL^2\left\|\ol{\mb{x}}^k - \mb{x}^*\right\|^2
		+ 4nL^2\mb{t}^k\right) \nonumber\\
		=&~ \left(1-\mu\alpha+\frac{4L^2\alpha^2}{n}\right)\left\|\ol{\mb x}^{k}-\mb{x}^*\right\|^2
		+ \frac{\alpha L^2}{n}\left(\frac{1}{\mu}+\frac{4\alpha}{n}\right)
		\left\|\mb{x}^k-\mb{1}_n\ol{\mb{x}}^k\right\|^2
		+ \frac{4L^2\alpha^2}{n}\mb{t}^k.
		\end{align*}
		If~$\alpha \leq \frac{n\mu}{8L^2}$,~$1-\mu\alpha+\frac{2L^2\alpha^2}{n}\leq 1-\frac{\mu\alpha}{2}$ and~$\frac{\alpha L^2}{n}\left(\frac{1}{\mu}+\frac{4\alpha}{n}\right)\leq\frac{3L^2\alpha}{2\mu n}$, which finishes the proof.
	\end{proof}
	
	Next, we derive an upper bound for the  gradient tracking error~$\mathbb{E}\left[\left\|\mb{y}^{k+1}-\mb{1}_n\ol{\mb{y}}^{k+1}\right\|^2 | \mc{F}^k\right]$.
	\begin{lem}\label{track_er_saga}
		If~$\alpha\leq\frac{\mu}{2L^2}$, then the following holds:
		\begin{align*}
		&\mathbb{E}\left[\left\|\mb{y}^{k+1} - W_{\infty}\mb{y}^{k+1}\right\|^2\Big|\mc{F}^k\right]
		\nonumber\\
		\leq&~\frac{104L^2}{1-\sigma^2}\left\|\mb{x}^k-\mb{1}_n\ol{\mb{x}}^k\right\|^2
		+ \frac{74nL^2}{1-\sigma^2}\left\|\ol{\mb{x}}^k-\mb{x}^*\right\|^2
		+ \left(\frac{1+\sigma^2}{2}+\frac{40L^2\alpha^2}{1-\sigma^2}\right)\mathbb{E}\left[\left\|\mb{y}^k - \mb{1}_n\ol{\mb{y}}^k\right\|^2\Big|\mc{F}^k\right]
		+ \frac{60nL^2}{1-\sigma^2}\mb{t}^k,
		\end{align*}
	\end{lem}
	\vspace{0.3cm}
	\begin{proof}
		Using the gradient tracking update, we have:
		\begin{align*}
		&\left\|\mb{y}^{k+1} - W_{\infty}\mb{y}^{k+1}\right\|^2
		\nonumber\\
		=&
		\left\|W\mb{y}^k +
		\mb{g}^{k+1}-\mb{g}^k-
		W_{\infty}\left(W\mb{y}^k +
		\mb{g}^{k+1}-\mb{g}^k\right)
		\right\|^2  =
		\left\|W\mb{y}^k-
		W_{\infty}\mb{y}^k + \left(I_n-
		W_{\infty}\right)\left(
		\mb{g}^{k+1}-\mb{g}^k\right)
		\right\|^2 \nonumber\\
		\leq&~
		\sigma^2 \left\|\mb{y}^k-
		W_{\infty}\mb{y}^k\right\|^2 + \left\|\mb{g}^{k+1}-\mb{g}^k\right\|^2
		+ 2\Big\langle W\mb{y}^k-
		W_{\infty}\mb{y}^k, \left(I_n-
		W_{\infty}\right)\left(
		\mb{g}^{k+1}-\mb{g}^k\right)\Big\rangle.
		\end{align*}
		We then take the conditional expectation given~$\mc{F}^k$ to obtain:
		\begin{align}\label{y1}
		\mathbb{E}\left[\left\|\mb{y}^{k+1} - \mb{1}_n\ol{\mb{y}}^{k+1}\right\|^2\Big|\mc{F}^k\right]
		\leq&~
		\sigma^2\mathbb{E}\left[\left\|\mb{y}^{k} - W_{\infty}\mb{y}^{k}\right\|^2\Big |\mc{F}^k\right]
		+ \mathbb{E}\left[\left\|\mb{g}^{k+1}-\mb{g}^k\right\|^2\Big|\mc{F}^k\right]\nonumber\\
		&+ 2\mathbb{E}\left[\Big\langle W\mb{y}^k-
		W_{\infty}\mb{y}^k, \left(I_n-
		W_{\infty}\right)\left(
		\mb{g}^{k+1}-\mb{g}^k\right)\Big\rangle\Big |\mc{F}^k\right]
		\end{align}
		Next we bound~$\mathbb{E}\left[\left\|\mb{g}^{k+1}-\mb{g}^k\right\|^2\Big|\mc{F}^k\right]$.
		\begin{align}
		&\mathbb{E}\left[\left\|\mb{g}^{k+1}-\mb{g}^k\right\|^2\Big|\mc{F}^k\right] = \mathbb{E}\left[\left\|\mb{g}^{k+1}-\mb{g}^k 
		-\left(\nabla\mb{f}(\mb{x}^{k+1})-\nabla\mb{f}(\mb{x}^{k})\right) +\left(\nabla\mb{f}(\mb{x}^{k+1})-\nabla\mb{f}(\mb{x}^{k})\right)
		\right\|^2\Big|\mc{F}^k\right] \nonumber\\
		=&~
		\mathbb{E}\left[\left\|\mb{g}^{k+1}-\mb{g}^k 
		-\left(\nabla\mb{f}(\mb{x}^{k+1})-\nabla\mb{f}(\mb{x}^{k})\right)\right\|^2\Big|\mc{F}^k\right]
		+
		\mathbb{E}\left[\left\|\nabla\mb{f}(\mb{x}^{k+1})-\nabla\mb{f}(\mb{x}^{k})\right\|^2\Big|\mc{F}^k\right]\nonumber\\
		&+
		2\mathbb{E}\left[\Big\langle \nabla\mb{f}(\mb{x}^{k+1})-\nabla\mb{f}(\mb{x}^{k})
		, \mb{g}^{k+1}-\mb{g}^k 
		-\left(\nabla\mb{f}(\mb{x}^{k+1})-\nabla\mb{f}(\mb{x}^{k})\right)\Big\rangle\Big|\mc{F}^k\right] \nonumber\\
		\triangleq&~ V_1 + V_2 + 2V_3.
		\end{align}
		Next, we bound~$V_1,V_2,V_3$ respectively, starting with~$V_2$.
		\begin{align}\label{gradf}
		&\left\|\nabla\mb{f}(\mb{x}^{k+1})-\nabla\mb{f}(\mb{x}^{k})\right\|^2\nonumber\\ 
		\leq&~ L^2\left\|\mb{x}^{k+1}-\mb{x}^{k}\right\|^2
		= L^2\left\|W\mb{x}^{k}-\alpha\mb{y}^k-\mb{x}^{k}\right\|^2
		\nonumber\\
		=&~ L^2\left\|\left(W-I_n\right)\left(\mb{x}^k-W_\infty\mb{x}^k\right)-\alpha\mb{y}^k\right\|^2
		\leq
		8L^2\left\|\mb{x}^k-W_\infty\mb{x}^k\right\|^2 + 2\alpha^2L^2\left\|\mb{y}^k\right\|^2.
		\end{align}
		Next we derive a bound for~$\left\|\mb{y}^k\right\|^2$. 
		\begin{align}
		\left\|\mb{y}^k\right\|
		=&~\left\|\mb{y}^k - W_\infty\mb{y}^k + W_\infty\mb{g}^k
		-W_\infty\nabla\mb{f}(\mb{x}^k) + W_\infty\nabla\mb{f}(\mb{x}^k)
		-W_\infty\nabla\mb{f}(\mb{1}_n\mb{x}^*)\right\| \nonumber\\
		\leq&~
		\left\|\mb{y}^k - W_\infty\mb{y}^k\right\|
		+ \sqrt{n} \left\|\ol{\mb{g}}^k-\mb{h}(\mb{x}^k)\right\| 
		+ L\left\|\mb{x}^k-\mb{1}_n\mb{x}^*\right\|
		\nonumber\\
		\leq&~
		\left\|\mb{y}^k - W_\infty\mb{y}^k\right\|
		+ L\left\|\mb{x}^k-\mb{1}_n\ol{\mb{x}}^k\right\|
		+ \sqrt{n}L\left\|\ol{\mb{x}}^k-\mb{x}^*\right\|
		+ \sqrt{n} \left\|\ol{\mb{g}}^k-\mb{h}(\mb{x}^k)\right\|.
		\nonumber
		\end{align}
		Squaring the last inequality above to obtain:
		\begin{align}\label{y}
		\left\|\mb{y}^k\right\|^2
		\leq 4\left\|\mb{y}^k - W_\infty\mb{y}^k\right\|^2
		+ 4L^2\left\|\mb{x}^k-\mb{1}_n\ol{\mb{x}}^k\right\|^2
		+ 4nL^2\left\|\ol{\mb{x}}^k-\mb{x}^*\right\|^2
		+ 4n\left\|\mb{h}(\mb{x}^k)-\ol{\mb{g}}^k\right\|^2
		\end{align}
		Using~\eqref{y} and Lemma~\ref{var} in~\eqref{gradf} obtains an upper bound on~$V_2$ as follows: 
		\begin{align}
		V_2\leq&
		\left(8L^2 + 8L^4\alpha^2\right)\left\|\mb{x}^k-\mb{1}_n\ol{\mb{x}}^k\right\|^2
		+ 8nL^4\alpha^2\left\|\ol{\mb{x}}^k-\mb{x}^*\right\|^2
		+ 8L^2\alpha^2\mathbb{E}\left[\left\|\mb{y}^k - \mb{1}_n\ol{\mb{y}}^k\right\|^2\Big |\mc{F}^k\right] \nonumber\\
		&+ 8nL^2\alpha^2\frac{1}{n^2}\left(4L^2\left\|\mb{x}^k-\mb{1}_n\ol{\mb{x}}^k\right\|^2
		+ 4nL^2\left\|\ol{\mb{x}}^k - \mb{x}^*\right\|^2
		+ 4nL^2\mb{t}^k\right)
		\end{align}
		If~$\alpha\leq\frac{1}{2L}$, then~$\alpha^2\leq\frac{1}{4L^2}$, we have the following:
		\begin{align}\label{V2}
		V_2\leq&~
		10L^2\left\|\mb{x}^k-\mb{1}_n\ol{\mb{x}}^k\right\|^2
		+ 2nL^2\left\|\ol{\mb{x}}^k-\mb{x}^*\right\|^2
		+ 8L^2\alpha^2\mathbb{E}\left[\left\|\mb{y}^k - \mb{1}_n\ol{\mb{y}}^k\right\|^2\Big |\mc{F}^k\right]\nonumber\\
		&+
		\frac{8L^2}{n}\left(\left\|\mb{x}^k-\mb{1}_n\ol{\mb{x}}^k\right\|^2
		+ n\left\|\ol{\mb{x}}^k - \mb{x}^*\right\|^2
		+ n\mb{t}^k\right)\nonumber\\
		\leq&~
		18L^2\left\|\mb{x}^k-\mb{1}_n\ol{\mb{x}}^k\right\|^2
		+ 10nL^2\left\|\ol{\mb{x}}^k-\mb{x}^*\right\|^2
		+ 8L^2\alpha^2\mathbb{E}\left[\left\|\mb{y}^k - \mb{1}_n\ol{\mb{y}}^k\right\|^2\Big |\mc{F}^k\right]
		+ 8L^2\mb{t}^k,
		\end{align}
		Next, we derive an upper bound for~$V_3$.
		\begin{align}\label{V3'}
		V_3 
		=&~ 
		\mathbb{E}\left[\mathbb{E}\left[\Big\langle \nabla\mb{f}(\mb{x}^{k+1})-\nabla\mb{f}(\mb{x}^{k})
		, \mb{g}^{k+1}-\mb{g}^k 
		-\left(\nabla\mb{f}(\mb{x}^{k+1})-\nabla\mb{f}(\mb{x}^{k})\right)\Big\rangle\Big|\mc{F}^{k+1}\right]\Big|\mc{F}^{k}\right]
		\nonumber\\
		=&~ 
		\mathbb{E}\left[\Big\langle \nabla\mb{f}(\mb{x}^{k+1})-\nabla\mb{f}(\mb{x}^{k})
		, \nabla\mb{f}(\mb{x}^{k})-\mb{g}^k 
		\Big\rangle\Big|\mc{F}_{k}\right]
		= \mathbb{E}\left[\Big\langle \nabla\mb{f}(\mb{x}^{k+1})
		, \nabla\mb{f}(\mb{x}^{k})-\mb{g}^k 
		\Big\rangle\Big|\mc{F}_{k}\right] \nonumber\\
		=&
		\sum_{i=1}^{n} \mathbb{E}\left[\Big\langle \nabla{f}_i(\mb{x}_i^{k+1})
		, \nabla{f}_i(\mb{x}^{k})-\mb{g}_i^k 
		\Big\rangle\Big|\mc{F}^{k}\right].
		\end{align}
		Note that 
		$$\nabla f_i(\mb{x}_i^{k+1})
		= \nabla f_i\left(\sum_{j=1}^{n}w_{ij}\mb{x}_{j}^k-\alpha\left(\sum_{j=1}^{n}w_{ij}\mb{y}_j^{k-1} + \mb{g}_i^k - \mb{g}_i^{k-1}\right)\right).
		$$
		We define~$\widehat{\nabla}_i^k$ as the following~\cite{DSGT_Pu}:
		$$
		\widehat{\nabla}_i^k
		\triangleq 
		\nabla f_i\left(\sum_{j=1}^{n}w_{ij}\mb{x}_{j}^k-\alpha\left(\sum_{j=1}^{n}w_{ij}\mb{y}_j^{k-1} + \nabla f_i(\mb{x}_i^{k}) - \mb{g}_i^{k-1}\right)\right).
		$$
		Therefore we have that 
		\begin{align}\label{PS}
		\left\|\nabla f_i(\mb{x}_i^{k+1})-\widehat{\nabla}_i^k\right\|
		\leq L\alpha\left\|\mb{g}_i^k-\nabla f_i(\mb{x}_i^{k})\right\|.
		\end{align}
		Using~\eqref{PS} in~\eqref{V3'}, we have the following: if~$\alpha\leq\frac{1}{2L}$:
		\begin{align}\label{V_3}
		V_3 
		=&~\sum_{i=1}^{n} \mathbb{E}\left[\Big\langle \nabla{f}_i(\mb{x}_i^{k+1})-\widehat{\nabla}_i^k
		, \nabla{f}_i(\mb{x}^{k})-\mb{g}_i^k 
		\Big\rangle\Big|\mc{F}_{k}\right] \nonumber\\
		\leq&~\alpha L\sum_{i=1}^{n}\mathbb{E}\left[\left\|\mb{g}_i^k-\nabla f_i(\mb{x}_i^{k})\right\|^2\Big|\mc{F}_k\right]\nonumber\\
		\leq&~\alpha L
		\left(4L^2\left\|\mb{x}^k-\mb{1}_n\ol{\mb{x}}^k\right\|^2
		+ 4nL^2\left\|\ol{\mb{x}}^k - \mb{x}^*\right\|^2
		+ 4nL^2\mb{t}^k\right)
		\nonumber\\
		\leq&~
		2L^2\left\|\mb{x}^k-\mb{1}_n\ol{\mb{x}}^k\right\|^2
		+ 2nL^2\left\|\ol{\mb{x}}^k - \mb{x}^*\right\|^2
		+ 2nL^2\mb{t}^k,
		\end{align}
		where in the second inequality we used Lemma~\ref{var}. Finally we derive an upper bound for~$V_1$.
		\begin{align}\label{V1'}
		V_1 =&~\mathbb{E}\left[\left\|\mb{g}^{k+1}-\mb{g}^k 
		-\left(\nabla\mb{f}(\mb{x}^{k+1})-\nabla\mb{f}(\mb{x}^{k})\right)\right\|^2\Big|\mc{F}^k\right]\nonumber\\
		\leq&~2\mathbb{E}\left[\left\|\mb{g}^{k+1}-\nabla\mb{f}(\mb{x}^{k+1})\right\|^2\Big|\mc{F}^k\right]
		+2\mathbb{E}\left[\left\|\mb{g}^k 
		-\nabla\mb{f}(\mb{x}^{k})\right\|^2\Big|\mc{F}^k\right]
		\nonumber\\
		=&~2\mathbb{E}\left[\mathbb{E}\left[\left\|\mb{g}^{k+1}-\nabla\mb{f}(\mb{x}^{k+1})\right\|^2\Big|\mc{F}^{k+1}\right]\Big|\mc{F}^{k}\right]
		+2\mathbb{E}\left[\left\|\mb{g}^k 
		-\nabla\mb{f}(\mb{x}^{k})\right\|^2\Big|\mc{F}^k\right]
		\end{align}
		We first bound~$\mathbb{E}\left[\mathbb{E}\left[\left\|\mb{g}^{k+1}-\nabla\mb{f}(\mb{x}^{k+1})\right\|^2\Big|\mc{F}^{k+1}\right]\Big|\mc{F}^{k}\right]$. Using~\eqref{var}, we have: if~$\alpha\leq\frac{\mu}{2L^2}$,
		\begin{align}
		&\mathbb{E}\left[\mathbb{E}\left[\left\|\mb{g}^{k+1}-\nabla\mb{f}(\mb{x}^{k+1})\right\|^2\Big|\mc{F}^{k+1}\right]\Big|\mc{F}^{k}\right] \nonumber\\
		\leq&~ 4L^2\mathbb{E}\left[\left\|\mb{x}^{k+1}-\mb{1}_n\ol{\mb{x}}^{k+1}\right\|^2\Big|\mc{F}^k\right]
		+ 4nL^2\mathbb{E}\left[\left\|\ol{\mb{x}}^{k+1} - \mb{x}^*\right\|^2\Big|\mc{F}^k\right]
		+ 4nL^2\mathbb{E}\left[\mb{t}^{k+1}\Big|\mc{F}^k\right].
		\end{align}
		We then apply Lemma~\ref{1},~\ref{2} and~\ref{3} to the above inequality to obtain: 
		\begin{align}\label{V1''}
		&\mathbb{E}\left[\mathbb{E}\left[\left\|\mb{g}^{k+1}-\nabla\mb{f}(\mb{x}^{k+1})\right\|^2\Big|\mc{F}^{k+1}\right]\Big|\mc{F}^{k}\right] \nonumber\\
		\leq&~
		4L^2\left(\frac{1+\sigma^2}{2}\left\|\mb{x}^k-\mb{1}_n\ol{\mb{x}}^k\right\|^2 + \frac{2\alpha^2}{1-\sigma^2}\mathbb{E}\left[\left\|\mb{y}^k-\mb{1}_n\ol{\mb{y}}^k\right\|^2\Big |\mc{F}^k\right]\right) \nonumber\\
		&+4nL^2\left(\left(1-\frac{\mu\alpha}{2}\right)\left\|\ol{\mb x}^{k}-\mb{x}^*\right\|^2
		+ \frac{3L^2\alpha}{2\mu n}
		\left\|\mb{x}^k-\mb{1}_n\ol{\mb{x}}^k\right\|^2
		+ \frac{4L^2\alpha^2}{n}\mb{t}^k\right)\nonumber\\
		&+ 4nL^2\left(
		\left(1-\frac{1}{M}\right)\mb{t}^k + \frac{2}{mn}\left\|\mb{x}^k-\mb{1}_n\ol{\mb{x}}^k\right\|^2
		+ \frac{2}{m}\left\|\ol{\mb{x}}^k-\mb{x}^*\right\|^2\right)
		\nonumber\\
		\leq&~15L^2\left\|\mb{x}^k-\mb{1}_n\ol{\mb{x}}^k\right\|^2 + 12nL^2\left\|\ol{\mb x}^{k}-\mb{x}^*\right\|^2
		+ \frac{8L^2\alpha^2}{1-\sigma^2}\mathbb{E}\left[\left\|\mb{y}^k-\mb{1}_n\ol{\mb{y}}^k\right\|^2\Big|\mc{F}^k \right]+ 8nL^2\mb{t}^k.
		\end{align}
		We use~\eqref{V1'},~\eqref{V1''} and Lemma~\ref{var} to obtain an upper bound for~$V_1$ as follows.
		\begin{align}\label{V1}
		V_1 \leq&~
		2\left(15L^2\left\|\mb{x}^k-\mb{1}_n\ol{\mb{x}}^k\right\|^2 + 12nL^2\left\|\ol{\mb x}^{k}-\mb{x}^*\right\|^2
		+ \frac{8L^2\alpha^2}{1-\sigma^2}\left\|\mb{y}^k-\mb{1}_n\ol{\mb{y}}^k\right\|^2 + 8nL^2\mb{t}^k\right)
		\nonumber\\
		&+ 2\left(4L^2\left\|\mb{x}^k-\mb{1}_n\ol{\mb{x}}^k\right\|^2
		+ 4nL^2\left\|\ol{\mb{x}}^k - \mb{x}^*\right\|^2
		+ 4nL^2\mb{t}^k\right)\nonumber\\
		=&~38L^2\left\|\mb{x}^k-\mb{1}_n\ol{\mb{x}}^k\right\|^2 + 32nL^2\left\|\ol{\mb x}^{k}-\mb{x}^*\right\|^2
		+ \frac{16L^2\alpha^2}{1-\sigma^2}\left\|\mb{y}^k-\mb{1}_n\ol{\mb{y}}^k\right\|^2 + 24nL^2\mb{t}^k
		\end{align}
		We apply the upper bounds on~$V_1,V_2,V_3$ in~\eqref{V1},~\eqref{V2} and~\eqref{V_3} to~\eqref{y1} to derive an upper bound for $\mathbb{E}\left[\left\|\mb{g}^{k+1}-\mb{g}^{k}\right\|^2\Big|\mc{F}^k\right]$.
		\begin{align}\label{g_difference}
		&\mathbb{E}\left[\left\|\mb{g}^{k+1}-\mb{g}^{k}\right\|^2\Big|\mc{F}^k\right] 
		\leq V_1 + V_2 + 2V_3 \nonumber\\
		\leq&~\left(38L^2\left\|\mb{x}^k-\mb{1}_n\ol{\mb{x}}^k\right\|^2 + 32nL^2\left\|\ol{\mb x}^{k}-\mb{x}^*\right\|^2
		+ \frac{16L^2\alpha^2}{1-\sigma^2}\mathbb{E}\left[\left\|\mb{y}^k - \mb{1}_n\ol{\mb{y}}^k\right\|^2\Big|\mc{F}^k\right] + 24nL^2\mb{t}^k\right) \nonumber\\
		&+ \left(18L^2\left\|\mb{x}^k-\mb{1}_n\ol{\mb{x}}^k\right\|^2
		+ 10nL^2\left\|\ol{\mb{x}}^k-\mb{x}^*\right\|^2
		+ 8L^2\alpha^2\mathbb{E}\left[\left\|\mb{y}^k - \mb{1}_n\ol{\mb{y}}^k\right\|^2\Big|\mc{F}^k\right]
		+ 8L^2\mb{t}^k\right)\nonumber\\
		&+ 2\left(2L^2\left\|\mb{x}^k-\mb{1}_n\ol{\mb{x}}^k\right\|^2
		+ 2nL^2\left\|\ol{\mb{x}}^k - \mb{x}^*\right\|^2
		+ 2nL^2\mb{t}^k\right)\nonumber\\
		=&~60L^2\left\|\mb{x}^k-\mb{1}_n\ol{\mb{x}}^k\right\|^2 + 46nL^2\left\|\ol{\mb x}^{k}-\mb{x}^*\right\|^2
		+ \frac{24L^2\alpha^2}{1-\sigma^2}\mathbb{E}\left[\left\|\mb{y}^k - \mb{1}_n\ol{\mb{y}}^k\right\|^2\Big|\mc{F}^k\right] + 36nL^2\mb{t}^k
		\end{align}
		Next, we derive an upper bound for~$2\mathbb{E}\left[\Big\langle W\mb{y}^k-
		W_{\infty}\mb{y}^k, \left(I_n-
		W_{\infty}\right)\left(
		\mb{g}^{k+1}-\mb{g}^k\right)\Big\rangle\Big |\mc{F}^k\right]$. We first note that:
		\begin{align*}
		\mathbb{E}\left[\Big\langle W\mb{y}^k-
		W_{\infty}\mb{y}^k, \left(I_n-
		W_{\infty}\right)\left(
		\mb{g}^{k+1}-\mb{g}^k\right)\Big\rangle\Big |\mc{F}^k\right]
		=
		\mathbb{E}\left[\Big\langle W\mb{y}^k-
		W_{\infty}\mb{y}^k, \mb{g}^{k+1}-\mb{g}^k\Big\rangle\Big |\mc{F}^k\right],
		\end{align*}
		since~$\Big\langle W\mb{y}^k-
		W_{\infty}\mb{y}^k, 
		W_{\infty}\left(
		\mb{g}^{k+1}-\mb{g}^k\right)\Big\rangle = 0$. Using the tower property of the conditional expectation,
		\begin{align}
		&2\mathbb{E}\left[\Big\langle W\mb{y}^k-
		W_{\infty}\mb{y}^k, \mb{g}^{k+1}-\mb{g}^k\Big\rangle\Big |\mc{F}^k\right]
		=
		2\mathbb{E}\left[\mathbb{E}\left[\Big\langle W\mb{y}^k-
		W_{\infty}\mb{y}^k, \mb{g}^{k+1}-\mb{g}^k\Big\rangle\Big |\mc{F}^{k+1}\right]\Big |\mc{F}^k\right]\nonumber\\
		=&~2\mathbb{E}\left[\Big\langle W\mb{y}^k-
		W_{\infty}\mb{y}^k, \nabla\mb{f}(\mb{x}^{k+1})-\mb{g}^k\Big\rangle\Big |\mc{F}^k\right]\nonumber\\
		=&~2\mathbb{E}\left[\Big\langle W\mb{y}^k-
		W_{\infty}\mb{y}^k, \nabla\mb{f}(\mb{x}^{k+1})-\nabla\mb{f}(\mb{x}^{k})\Big\rangle\Big |\mc{F}^k\right]
		+ 2\mathbb{E}\left[\Big\langle W\mb{y}^k-
		W_{\infty}\mb{y}^k, \nabla\mb{f}(\mb{x}^{k})-\mb{g}^k\Big\rangle\Big |\mc{F}^k\right]\nonumber\\
		\triangleq&~R_1 + R_2
		\end{align}
		Next, we bound~$R_1$ and~$R_2$ separately, starting with~$R_1$.
		\begin{align*}
		2\Big\langle W\mb{y}^k-
		W_{\infty}\mb{y}^k, \nabla\mb{f}(\mb{x}^{k+1})-\nabla\mb{f}(\mb{x}^{k})\Big\rangle 
		\leq\frac{1-\sigma^2}{2}\left\|\mb{y}^k-
		W_{\infty}\mb{y}^k\right\|^2 + \frac{2\sigma^2}{1-\sigma^2} \left\|\nabla\mb{f}(\mb{x}^{k+1})-\nabla\mb{f}(\mb{x}^{k})\right\|^2.
		\end{align*}
		Taking the conditional expectation given~$\mc{F}^k$ and Applying~\eqref{V2} to the above inequality, we have that:
		\begin{align}\label{I_1}
		R_1 \leq&~\frac{1-\sigma^2}{2}\mathbb{E}\left[\left\|\mb{y}^k-
		\mb{1}_n\ol{\mb{y}}^k\right\|^2\Big |\mc{F}^k\right]\nonumber\\
		&+ \frac{2\sigma^2}{1-\sigma^2}\left(18L^2\left\|\mb{x}^k-\mb{1}_n\ol{\mb{x}}^k\right\|^2
		+ 10nL^2\left\|\ol{\mb{x}}^k-\mb{x}^*\right\|^2
		+ 8L^2\alpha^2\mathbb{E}\left[\left\|\mb{y}^k - \mb{1}_n\ol{\mb{y}}^k\right\|^2\Big|\mc{F}^k\right]
		+ 8L^2\mb{t}^k\right) \nonumber\\
		\leq&~\frac{36L^2}{1-\sigma^2}\left\|\mb{x}^k-\mb{1}_n\ol{\mb{x}}^k\right\|^2
		+ \frac{20nL^2}{1-\sigma^2}\left\|\ol{\mb{x}}^k-\mb{x}^*\right\|^2
		+ \frac{16L^2}{1-\sigma^2}\mb{t}^k \nonumber\\
		&\qquad\qquad+ \left(\frac{1-\sigma^2}{2}+\frac{16L^2\alpha^2}{1-\sigma^2}\right)\mathbb{E}\left[\left\|\mb{y}^k - \mb{1}_n\ol{\mb{y}}^k\right\|^2\Big|\mc{F}^k\right]
		\end{align}
		Next, we bound~$R_2$. We first note that:
		\begin{align*}
		R_2 = 2\mathbb{E}\left[\Big\langle W\mb{y}^k, \nabla\mb{f}(\mb{x}^{k})-\mb{g}^k\Big\rangle\Big |\mc{F}^k\right]
		-
		2\mathbb{E}\left[\Big\langle
		W_{\infty}\mb{y}^k, \nabla\mb{f}(\mb{x}^{k})-\mb{g}^k\Big\rangle\Big |\mc{F}^k\right].
		\end{align*}
		For the first term, using the~$\mb{y}^k$-update of the algorithm, we have that
		\begin{align*}
		2\mathbb{E}\left[\Big\langle W\mb{y}^k, \nabla\mb{f}(\mb{x}^{k})-\mb{g}^k\Big\rangle\Big |\mc{F}^k\right]
		=&~2\mathbb{E}\left[\Big\langle W^2\mb{y}^{k-1}+\mb{g}^{k}-\mb{g}^{k-1}, \nabla\mb{f}(\mb{x}^{k})-\mb{g}^k\Big\rangle\Big |\mc{F}^k\right] \nonumber\\
		=&~2\mathbb{E}\left[\Big\langle \mb{g}^{k}, \nabla\mb{f}(\mb{x}^{k})-\mb{g}^k\Big\rangle\Big |\mc{F}^k\right]
		=2\mathbb{E}\left[\Big\langle \mb{g}^{k}-\nabla\mb{f}(\mb{x}^{k}), \nabla\mb{f}(\mb{x}^{k})-\mb{g}^k\Big\rangle\Big |\mc{F}^k\right]\leq 0.
		\end{align*}
		For the second term, since~$\{\mb{g}_i^k\}$ are independent given~$\mc{F}^k$, we have 
		\begin{align}\label{I_2}
		&-2\mathbb{E}\left[\Big\langle
		W_{\infty}\mb{y}^k, \nabla\mb{f}(\mb{x}^{k})-\mb{g}^k\Big\rangle\Big |\mc{F}^k\right] \nonumber\\
		=&-2\sum_{i=1}^{n}\mathbb{E}\left[\Big\langle
		\ol{\mb{g}}^k, \nabla f_i(\mb{x}_i^{k})-\mb{g}_i^k\Big\rangle\Big |\mc{F}^k\right]
		= -\frac{2}{n}\sum_{i=1}^{n}\mathbb{E}\left[\Big\langle
		\mb{g}_i^k, \nabla f_i(\mb{x}_i^{k})-\mb{g}_i^k\Big\rangle\Big |\mc{F}^k\right]
		\nonumber\\
		=&~\frac{2}{n}\mathbb{E}\left[
		\left\|\mb{g}^k-\nabla\mb{f}(\mb{x}^{k})\right\|^2
		\Big |\mc{F}^k\right]
		\leq \frac{8L^2}{n}\left\|\mb{x}^k-\mb{1}_n\ol{\mb{x}}^k\right\|^2
		+ 8L^2\left\|\ol{\mb{x}}^k - \mb{x}^*\right\|^2
		+ 8L^2\mb{t}^k,
		\end{align}
		where in the last inequality we used~\eqref{var}. Combining the upper bounds on~$R_1$ and~$R_2$ in~\eqref{I_1} and~\eqref{I_2}, we have:
		\begin{align}\label{y_cross}
		&\mathbb{E}\left[\Big\langle W\mb{y}^k-
		W_{\infty}\mb{y}^k, \mb{g}^{k+1}-\mb{g}^k\Big\rangle\Big |\mc{F}^k\right]
		\nonumber\\
		\leq&~
		\frac{44L^2}{1-\sigma^2}\left\|\mb{x}^k-\mb{1}_n\ol{\mb{x}}^k\right\|^2
		+ \frac{28nL^2}{1-\sigma^2}\left\|\ol{\mb{x}}^k-\mb{x}^*\right\|^2
		+ \frac{24L^2}{1-\sigma^2}\mb{t}^k\nonumber\\
		&\qquad\qquad+ \left(\frac{1-\sigma^2}{2}+\frac{16L^2\alpha^2}{1-\sigma^2}\right)\mathbb{E}\left[\left\|\mb{y}^k - \mb{1}_n\ol{\mb{y}}^k\right\|^2\Big|\mc{F}^k\right].
		\end{align}
		Finally we combine~\eqref{y1},~\eqref{g_difference} and~\eqref{y_cross} to obtain an upper bound for~$\mathbb{E}\left[\left\|\mb{y}^{k+1} - W_{\infty}\mb{y}^{k+1}\right\|^2\Big|\mc{F}^k\right]$.
		\begin{align*}
		&\mathbb{E}\left[\left\|\mb{y}^{k+1} - \mb{1}_n\ol{\mb{y}}^{k+1}\right\|^2\Big|\mc{F}^k\right]
		\nonumber\\
		\leq&~
		\sigma^2\mathbb{E}\left[\left\|\mb{y}^{k} - W_{\infty}\mb{y}^{k}\right\|^2\Big|\mc{F}^k\right]
		\nonumber\\
		&+ 60L^2\left\|\mb{x}^k-\mb{1}_n\ol{\mb{x}}^k\right\|^2 + 46nL^2\left\|\ol{\mb x}^{k}-\mb{x}^*\right\|^2
		+ \frac{24L^2\alpha^2}{1-\sigma^2}\left\|\mb{y}^k-\mb{1}_n\ol{\mb{y}}^k\right\|^2 + 36nL^2\mb{t}^k\nonumber\\
		&+ \frac{44L^2}{1-\sigma^2}\left\|\mb{x}^k-\mb{1}_n\ol{\mb{x}}^k\right\|^2
		+ \frac{28nL^2}{1-\sigma^2}\left\|\ol{\mb{x}}^k-\mb{x}^*\right\|^2
		+ \left(\frac{1-\sigma^2}{2}+\frac{16L^2\alpha^2}{1-\sigma^2}\right)\mathbb{E}\left[\left\|\mb{y}^k - \mb{1}_n\ol{\mb{y}}^k\right\|^2\Big|\mc{F}^k\right]
		+ \frac{24L^2}{1-\sigma^2}\mb{t}^k\nonumber\\
		\leq&~
		\frac{104L^2}{1-\sigma^2}\left\|\mb{x}^k-\mb{1}_n\ol{\mb{x}}^k\right\|^2
		+ \frac{74nL^2}{1-\sigma^2}\left\|\ol{\mb{x}}^k-\mb{x}^*\right\|^2
		+ \left(\frac{1+\sigma^2}{2}+\frac{40L^2\alpha^2}{1-\sigma^2}\right)\mathbb{E}\left[\left\|\mb{y}^k - \mb{1}_n\ol{\mb{y}}^k\right\|^2\Big|\mc{F}^k\right]
		+ \frac{60nL^2}{1-\sigma^2}\mb{t}^k,
		\end{align*}
		which finishes the proof.
	\end{proof}
	\subsection{Main Results}
	With the help of previous Lemmas, we derive the range of the step-size~$\alpha$ where~\textbf{\texttt{GT-SAGA}} achieves linear convergence. Recall Lemma~\ref{1}-\ref{4} and take total expectation of these inequalities to obtain: 
	\begin{align}
	&\mathbb{E}\left[\left\|\mb{x}^{k+1}-\mb{1}_n\ol{\mb{x}}^{k+1}\right\|^2\right]\leq
	\frac{1+\sigma^2}{2}\mathbb{E}\left[\left\|\mb{x}^k-\mb{1}_n\ol{\mb{x}}^k\right\|^2\right] + \frac{2\alpha^2}{1-\sigma^2}\mathbb{E}\left[\left\|\mb{y}^k-\mb{1}_n\ol{\mb{y}}^k\right\|^2. \label{1}\right]\\
	&\mathbb{E}\left[n\left\|\ol{\mb x}^{k+1}-\mb{x}^*\right\|^2\right]
	\leq
	\frac{3L^2\alpha}{2\mu}
	\mathbb{E}\left[\left\|\mb{x}^k-\mb{1}_n\ol{\mb{x}}^k\right\|^2\right]
	+ \left(1-\frac{\mu\alpha}{2}\right)\mathbb{E}\left[n\left\|\ol{\mb x}^{k}-\mb{x}^*\right\|^2\right]
	+ \frac{4L^2\alpha^2}{n}\mathbb{E}\left[n\mb{t}^k\right].
	\label{2}\\
	&\mathbb{E}\left[n\mb{t}^{k+1}\right]\leq \frac{2}{m}\mathbb{E}\left[\left\|\mb{x}^k-\mb{1}_n\ol{\mb{x}}^k\right\|^2\right]
	+ \frac{2}{m}\mathbb{E}\left[n\left\|\ol{\mb{x}}^k-\mb{x}^*\right\|^2\right]
	+ \left(1-\frac{1}{M}\right)\mathbb{E}\left[n\mb{t}^k\right].
	\label{3}\\
	&\mathbb{E}\left[\left\|\mb{y}^{k+1} - \mb{1}_n\ol{\mb{y}}^{k+1}\right\|^2\right]
	\leq\frac{104L^2}{1-\sigma^2}\mathbb{E}\left[\left\|\mb{x}^k-\mb{1}_n\ol{\mb{x}}^k\right\|^2\right]
	+ \frac{74L^2}{1-\sigma^2}\mathbb{E}\left[n\left\|\ol{\mb{x}}^k-\mb{x}^*\right\|^2\right]
	+
	\frac{60L^2}{1-\sigma^2}\mathbb{E}\left[n\mb{t}^k\right] 
	\nonumber\\
	&\qquad\qquad\qquad\qquad\qquad\qquad
	+
	\left(\frac{1+\sigma^2}{2}+\frac{40L^2\alpha^2}{1-\sigma^2}\right)\mathbb{E}\left[\left\|\mb{y}^k - \mb{1}_n\ol{\mb{y}}^k\right\|^2\right].\label{4}
	\end{align}
	Now, we write~\eqref{1}-\eqref{4} in the form of a linear system as follows.
	$$\mb{u}^{k+1}\leq J_\alpha\mb{u}^{k},$$
	where~$\mb{u}^k\in\mathbb{R}^4$ and~$J_\alpha\in\mathbb{R}^{4\times 4}$ are given below.
	\begin{align}
	\mb{u}^k=\left[
	\begin{array}{l}
	\mathbb{E}\left[\left\|\mb{x}^{k}-\mb{1}_n\ol{\mb{x}}^{k}\right\|^2\right] \\
	\mathbb{E}\left[n\left\|\ol{\mb x}^{k}-\mb{x}^*\right\|^2\right] \\
	\mathbb{E}\left[n\mb{t}^{k}\right] \\
	\mathbb{E}\left[\left\|\mb{y}^{k} - \mb{1}\ol{\mb{y}}^{k}\right\|^2\right]
	\end{array}
	\right],\qquad
	J_\alpha=\left[
	\begin{array}{cccc}
	\frac{1+\sigma^2}{2} & 0 & 0 &\frac{2\alpha^2}{1-\sigma^2} \\
	\frac{3L^2\alpha}{2\mu}&  1-\frac{\mu\alpha}{2} & 4L^2\alpha^2 & 0\\
	\frac{2}{m} & \frac{2}{m} & 1-\frac{1}{M} &0\\
	\frac{104L^2}{1-\sigma^2} & \frac{74L^2}{1-\sigma^2} & \frac{60L^2}{1-\sigma^2}  & \frac{1+\sigma^2}{2} + \frac{40L^2\alpha^2}{1-\sigma^2}
	\end{array}
	\right].
	\end{align}
	Next, we derive the range of~$\alpha$ such that~$\rho\left(J_\alpha\right)<1$. To do that, we present the following Lemma from~\cite{matrix_analysis}. For the sake of completeness, we also give its proof here.
	\begin{lem}\label{rho_bound}
		Let~$A\in\mathbb{R}^{d\times d}$ be a non-negative matrix and~$\mb{x}\in\mathbb{R}^d$ be a positive vector. If~$A\mb{x}\leq\beta\mb{x}$ for~$\beta>0$, then~$\rho(A)\leq\beta$. If~$A\mb{x}<\gamma\mb{x}$ for~$\gamma>0$, then~$\rho(A)<\gamma$.  
	\end{lem}
	\begin{proof}
		We use~$x_i$ to denote the~$i$th entry of~$\mb{x}$. If~$A\mb{x}\leq\beta\mb{x}$, then~$\sum_{j=1}^{d}a_{ij}x_i\leq\beta x_i$,~$\forall i\in\{1,\cdots,d\}$. Define~$S \triangleq \mbox{diag}\{x_1,\cdots,x_d\}$. Then we have, 
		\begin{align*}
		\beta \geq \max_{i\in\{1,\cdots,d\}}\sum_{j=1}^{d}x_i^{-1}a_{ij}x_j
		= \mn{S^{-1}AS}_\infty\geq\rho\left(S^{-1}AS\right)
		=\rho\left(A\right),
		\end{align*} 
		where~$\mn{\cdot}_\infty$ denotes the matrix norm of maximum row sum. If~$A\mb{x}<\gamma\mb{x}$,~$\exists\gamma'>0$, such that~$\gamma'<\gamma$ and~$A\mb{x}\leq\gamma'\mb{x}$. Therefore,~$\rho\left(A\right)\leq\gamma'<\gamma$.
	\end{proof}
	\begin{theorem}
		If the step-size~$\alpha$ satisfies~$0<\alpha < \frac{m}{M}\frac{\left(1-\sigma^2\right)^2}{140QL},$
		then~\textbf{\texttt{GT-SAGA}} is linearly convergent. Moreover, if~$\alpha = \frac{m}{M}\frac{\left(1-\sigma^2\right)^2}{150QL}$,~\textbf{\texttt{GT-SAGA}} achieves~$\epsilon$-accuracy in~$$O\left(\max\left\{M,\frac{m}{M}\frac{Q^2}{\left(1-\sigma\right)^2}\right\}\log\frac{1}{\epsilon}\right)$$
		iterations (local component gradient computations), where~$m$ and~$M$ are respectively the minimum and maximum number of local functions at all nodes.
	\end{theorem}
	\begin{proof}
		In the light of Lemma~\ref{rho_bound}, we solve for the range of the step-size~$\alpha$ and a positive vector~$\bds\epsilon = \left[\epsilon_1,\epsilon_2,\epsilon_3,\epsilon_4\right]^\top$ such that the following (entry-wise) inequality holds for some~$\kappa>1$.
		\begin{align*}
		J_\alpha\bds{\epsilon} < \left(1-\frac{1}{\kappa}\right)\bds\epsilon.
		\end{align*}
		We expand the above matrix-vector inequality as follows.
		\begin{align}
		\frac{1+\sigma^2}{2}\epsilon_1 + \frac{2\alpha^2}{1-\sigma^2}\epsilon_4 <&~ \left(1-\frac{1}{\kappa}\right)\epsilon_1.
		\label{r1}\\
		\frac{3L^2\alpha}{2\mu}\epsilon_1 + \left(1-\frac{\mu\alpha}{2}\right)\epsilon_2 + 4L^2\alpha^2\epsilon_3 <&~ \left(1-\frac{1}{\kappa}\right)\epsilon_2.
		\label{r2}\\
		\frac{2}{m}\epsilon_1 + \frac{2}{m}\epsilon_2
		+ \left(1-\frac{1}{M}\right)\epsilon_3 <&~ \left(1-\frac{1}{\kappa}\right)\epsilon_3.
		\label{r3}\\
		\frac{104L^2}{1-\sigma^2}\epsilon_1 
		+ \frac{74L^2}{1-\sigma^2}\epsilon_2
		+ \frac{60L^2}{1-\sigma^2}\epsilon_3
		+ \frac{1+\sigma^2}{2}\epsilon_4
		+ \frac{40L^2\alpha^2}{1-\sigma^2}\epsilon_4 <&~ \left(1-\frac{1}{\kappa}\right)\epsilon_4.\label{r4}
		\end{align}
		Then we rewrite the above inequalities in the following form:
		\begin{align}
		\frac{1}{\kappa}\leq&~\frac{1-\sigma^2}{2} - \frac{2\alpha^2}{1-\sigma^2}\frac{\epsilon_4}{\epsilon_1}
		\label{r1'}\\
		\frac{1}{\kappa}\leq&~
		\frac{\mu\alpha}{2} - \frac{3L^2\alpha}{2\mu}\frac{\epsilon_1}{\epsilon_2}
		- \frac{4L^2\alpha^2\epsilon_3}{\epsilon_2} \label{r2'}\\
		\frac{1}{\kappa}
		\leq&~\frac{1}{M} - \frac{2}{m}\frac{\epsilon_1}{\epsilon_3}
		- \frac{2}{m}\frac{\epsilon_2}{\epsilon_3} \label{r3'}\\
		\frac{1}{\kappa}\leq&~
		\frac{1-\sigma^2}{2}
		-\frac{104L^2}{1-\sigma^2}\frac{\epsilon_1}{\epsilon_4}
		- \frac{74L^2}{1-\sigma^2}\frac{\epsilon_2}{\epsilon_4}
		- \frac{60L^2}{1-\sigma^2}\frac{\epsilon_3}{\epsilon_4}
		- \frac{40L^2\alpha^2}{1-\sigma^2}\label{r4'}
		\end{align}
		It is straightforward to see that the requirement that~\eqref{r1'}-\eqref{r4'} hold for some~$\kappa>1$ is equivalent to the RHS of~\eqref{r1'}-\eqref{r4'} being positive. Next, we fix the positive vector~$\bds\epsilon$ that is independent of~$\alpha$ and~$\kappa$. The RHS of~\eqref{r2'} being positive is equivalent to the following:
		\begin{align}
		4L^2\epsilon_3\alpha < \frac{\mu}{2}\epsilon_2 
		- \frac{3L^2}{2\mu}\epsilon_1.\label{r2''}
		\end{align}
		We set~$\epsilon_1 = 1$ and~$\epsilon_2 = 4Q$, where~$Q=L/\mu$.
		The RHS of~\eqref{r3'} being positive is equivalent to the following:
		\begin{align}
		\epsilon_3 > \frac{2M}{m}\epsilon_1 + \frac{2M}{m}\epsilon_2
		= \frac{2M}{m} + \frac{8MQ^2}{m},
		\end{align}
		where we used the previously fixed values of~$\epsilon_1$ and~$\epsilon_2$. We therefore set~$\epsilon_3 = \frac{12MQ^2}{m}$. Finally, we note that the RHS of~\eqref{r4'} being positive is equivalent to the following:
		\begin{align}\label{r4''}
		\frac{40L^2\alpha^2}{1-\sigma^2}\epsilon_4 < 
		\frac{1-\sigma^2}{2}\epsilon_4
		- \frac{104L^2}{1-\sigma^2}\epsilon_1 
		- \frac{74L^2}{1-\sigma^2}\epsilon_2
		- \frac{60L^2}{1-\sigma^2}\epsilon_3.
		\end{align}
		For the RHS of~\eqref{r4''} to be positive, 
		\begin{align*}
		\epsilon_4 >
		\frac{2L^2}{\left(1-\sigma^2\right)^2}
		\left(104 + 296Q^2 + \frac{720MQ^2}{m}\right),
		\end{align*}
		where we used the previously fixed values of~$\epsilon_1,\epsilon_2$ and~$\epsilon_3$. Since~$104 + 296Q^2 + \frac{720MQ^2}{m} < \frac{1120MQ^2}{m}$, we set~$\epsilon_4 = \frac{2250}{\left(1-\sigma^2\right)^2}\frac{ML^2Q^2}{m}$. So far, we have fixed the values of~$\epsilon_1,\epsilon_2,\epsilon_3$ and~$\epsilon_4$ as the following:
		\begin{align}
		\epsilon_1 = 1,\quad\epsilon_2 = 4Q,\quad\epsilon_3 = \frac{12MQ^2}{m},\quad\epsilon_4 = \frac{2250}{\left(1-\sigma^2\right)^2}\frac{ML^2Q^2}{m}.
		\end{align}
		Now, we find the range of~$\alpha$ from~\eqref{r1'},~\eqref{r2''} and~\eqref{r4''}. 
		For the RHS of~\eqref{r1'} to be positive, we have that:
		\begin{align}
		\alpha < \sqrt{\frac{\left(1-\sigma^2\right)^2}{4}\frac{\epsilon_1}{\epsilon_4}}
		= \sqrt{\frac{\left(1-\sigma^2\right)^2}{4}\frac{\left(1-\sigma^2\right)^2}{2250}\frac{m}{ML^2Q^2}}
		= \frac{\left(1-\sigma^2\right)^2}{30\sqrt{10}}\frac{\sqrt{m}}{\sqrt{M}LQ} \label{a1}.
		\end{align}
		From~\eqref{r2''}, we have that
		\begin{align}
		\alpha < \frac{1}{8\mu\epsilon_3} = \frac{1}{8\mu}\frac{m}{12MQ^2} = \frac{m}{96M}\frac{1}{QL}
		\label{a2}
		\end{align}
		Finally, from~\eqref{r4''}, we have that:	
		\begin{align}\label{a4}
		&~\frac{40L^2\alpha^2}{1-\sigma^2}\epsilon_4 < 
		\frac{1-\sigma^2}{2}\epsilon_4
		- \left(\frac{104L^2}{1-\sigma^2}\epsilon_1 
		+ \frac{74L^2}{1-\sigma^2}\epsilon_2
		+ \frac{60L^2}{1-\sigma^2}\epsilon_3 \right). \nonumber\\
		\Longleftarrow&~
		\frac{40L^2\alpha^2}{1-\sigma^2}\epsilon_4 < 
		\frac{1-\sigma^2}{2}\epsilon_4
		- \frac{1120ML^2Q^2}{m(1-\sigma^2)} \nonumber\\
		\iff&
		\frac{40L^2}{1-\sigma^2}\alpha^2
		< \frac{1-\sigma^2}{2} - \frac{1120ML^2Q^2}{m(1-\sigma^2)}
		\frac{\left(1-\sigma^2\right)^2}{2250}\frac{m}{ML^2Q^2}
		=\frac{1-\sigma^2}{450} 
		\iff
		\alpha < \frac{1-\sigma^2}{60\sqrt{5}L}
		\end{align}	
		Therefore, from~\eqref{a1},~\eqref{a2} and~\eqref{a4}, we have that if~$\alpha$ satisfies:
		\begin{align*}
		&~\alpha < \ol\alpha\triangleq\min\left\{\frac{\left(1-\sigma^2\right)^2}{30\sqrt{10}}\frac{\sqrt{m}}{\sqrt{M}LQ},\frac{m}{96M}\frac{1}{QL},\frac{1-\sigma^2}{60\sqrt{5}L}\right\}, \nonumber\\
		\Longleftarrow&~
		\alpha < \frac{m}{M}\frac{\left(1-\sigma^2\right)^2}{140QL},
		\end{align*}
		there exists a sufficiently large~$\kappa_\alpha$ such that~\eqref{r1'}-\eqref{r4'} hold with~$\bds\epsilon =\left[1,4Q,\frac{12MQ^2}{m},\frac{2250}{\left(1-\sigma^2\right)^2}\frac{ML^2Q^2}{m}\right]^\top$, i.e, the algorithm is linearly convergent. Next, we derive an explicit convergence rate when we set~$\alpha = \frac{m}{M}\frac{\left(1-\sigma^2\right)^2}{150QL}$, which is slightly smaller than~$\ol\alpha$.
		From~\eqref{r1'}, we have that:
		\begin{align}\label{k1}
		&~\frac{1}{\kappa}\leq \frac{1-\sigma^2}{2} - \frac{2\alpha^2}{1-\sigma^2}\frac{\epsilon_4}{\epsilon_1}
		= \frac{1-\sigma^2}{2} - \frac{2}{1-\sigma^2}\frac{m^2}{M^2}\frac{\left(1-\sigma^2\right)^4}{150^2Q^2L^2}\frac{2250}{\left(1-\sigma^2\right)^2}\frac{ML^2Q^2}{m} = \frac{1-\sigma^2}{2} - \frac{1-\sigma^2}{5}\frac{m}{M}\nonumber\\
		\Longleftarrow&~
		\frac{1}{\kappa}\leq \frac{3\left(1-\sigma^2\right)}{10}.
		\end{align}
		From~\eqref{r2'}, we have that:
		\begin{align}\label{k2}
		\frac{1}{\kappa}
		&~\leq\frac{\mu\alpha}{2} - \frac{3L^2\alpha}{2\mu}\frac{\epsilon_1}{\epsilon_2}
		- \frac{4L^2\alpha^2\epsilon_3}{\epsilon_2}\nonumber\\
		&~=
		\frac{\mu}{2}\frac{m}{M}\frac{\left(1-\sigma^2\right)^2}{150QL} - \frac{3L^2}{2\mu}\frac{m}{M}\frac{\left(1-\sigma^2\right)^2}{150QL}\frac{1}{4Q}
		- \frac{4L^2}{4Q}\frac{12MQ^2}{m}\frac{m^2}{M^2}\frac{\left(1-\sigma^2\right)^4}{150^2Q^2L^2} \nonumber\\
		&~=\frac{m}{M}\frac{\left(1-\sigma^2\right)^2}{300Q^2}
		- \frac{m}{M}\frac{\left(1-\sigma^2\right)^2}{400Q}
		- \frac{m}{M}\frac{12\left(1-\sigma^2\right)^4}{150^2Q}
		\nonumber\\
		\Longleftarrow&~\frac{1}{\kappa}\leq\frac{m}{M}\frac{\left(1-\sigma^2\right)^2}{1200Q^2}
		- \frac{m}{M}\frac{\left(1-\sigma^2\right)^2}{1875}
		= \frac{3}{1000}\frac{m}{M}\frac{\left(1-\sigma^2\right)^2}{Q^2}
		\end{align}
		From~\eqref{r2'}, we have that:
		\begin{align}\label{k3}
		&~\frac{1}{\kappa}
		\leq\frac{1}{M} - \frac{2}{m}\frac{\epsilon_1}{\epsilon_3}
		- \frac{2}{m}\frac{\epsilon_2}{\epsilon_3}
		= \frac{1}{M} - \frac{2}{m}\frac{m}{12MQ^2}
		-\frac{2}{m}\frac{4Qm}{12MQ^2}
		= \frac{1}{M} -\frac{1}{6MQ^2} - \frac{2}{3MQ^2}
		\nonumber\\
		\Longleftarrow&~
		\frac{1}{\kappa}\leq\frac{1}{6M}
		\end{align}	
		Finally, from~\eqref{r4'} we have that:
		\begin{align}\label{k4}
		&\frac{1}{\kappa}\leq
		\frac{1-\sigma^2}{2}
		-\left(\frac{104L^2}{1-\sigma^2}\epsilon_1
		+ \frac{74L^2}{1-\sigma^2}\epsilon_2
		+ \frac{60L^2}{1-\sigma^2}\epsilon_3\right)\frac{1}{\epsilon_4}
		- \frac{40L^2\alpha^2}{1-\sigma^2}\nonumber\\
		\Longleftarrow&
		\frac{1}{\kappa}\leq\frac{1-\sigma^2}{2}
		- \frac{1120MQ^2L^2}{m\left(1-\sigma^2\right)}\frac{\left(1-\sigma^2\right)^2}{2250}\frac{m}{ML^2Q^2}
		- \frac{40L^2}{1-\sigma^2}\frac{m^2}{M^2}\frac{\left(1-\sigma^2\right)^4}{150^2Q^2L^2}\nonumber\\
		\iff&
		\frac{1}{\kappa}\leq\frac{1-\sigma^2}{2}
		- \frac{1120\left(1-\sigma^2\right)}{2250}
		- \frac{4}{2250}\frac{m^2}{M^2}\frac{\left(1-\sigma^2\right)^3}{Q^2}\nonumber\\
		\Longleftarrow&~
		\frac{1}{\kappa}\leq\frac{1-\sigma^2}{2250}
		\end{align}	
		Therefore, from~\eqref{k1}-\eqref{k4}, we have:
		\begin{align*}
		\frac{1}{\kappa}\leq\min\left\{\frac{3\left(1-\sigma^2\right)}{10},~\frac{3}{1000}\frac{m}{M}\frac{\left(1-\sigma^2\right)^2}{Q^2},~\frac{1}{6M},~\frac{1-\sigma^2}{2250}\right\},
		\end{align*}
		which completes the proof.
	\end{proof}

	\bibliographystyle{IEEEbib}
	\bibliography{sample1.bib}
\end{document}